\newcount\Comments
\pdfoutput=1
\documentclass[letterpaper,10pt]{scrartcl}

\usepackage[utf8]{inputenc}
\usepackage[english]{babel}
\usepackage[automark]{scrlayer-scrpage}
\usepackage{amsmath,amstext,amsbsy,amsopn,amscd,amsxtra,upref,amssymb}
\usepackage{amsfonts,bm}
\usepackage{amsthm}
\usepackage[numbers]{natbib}
\usepackage{color}
\usepackage{graphicx}
\usepackage{sidecap}
\usepackage{multirow}
\usepackage{booktabs}
\usepackage{bm}
\usepackage{fullpage}

\usepackage{tabularx}
\usepackage[font=small]{caption}
\usepackage{subfig}

\usepackage{longtable}
\usepackage{rotating}

\usepackage{algorithm}
\usepackage{algorithmicx}
\usepackage[algo2e,ruled]{algorithm2e}
\extrafloats{100}

\usepackage{lineno,hyperref}

\graphicspath{{figures/}}
\DeclareGraphicsExtensions{.pdf,.eps,.png,.jpg,.jpeg}

\DeclareOldFontCommand{\bf}{\normalfont\bfseries}{\mathbf}

\usepackage{bbm}

\newcommand{\bGamma}{{\boldsymbol \Gamma}}
\newcommand{\bA}{{\bf A}}
\newcommand{\bu}{{\boldsymbol u}}
\newcommand{\bv}{{\boldsymbol v}}
\newcommand{\bw}{{\boldsymbol w}}
\newcommand{\bg}{{\boldsymbol{g}}}
\newcommand{\bphi}{{\boldsymbol \phi}}

\newcommand{\KL}{\operatorname{KL}}
\newcommand{\dhell}{\operatorname{d}_{\mathrm{Hell}}}

\newcommand{\bE}{\mathbb{E}}
\newcommand {\R} {{\mathbb{R}}}
\renewcommand {\P} {{\mathcal P}}
\newcommand {\Y} {{\mathcal Y}}
\newcommand {\bx} {{\mathbf{x}}}
\newcommand {\by} {{\mathbf{y}}}
\newcommand{\btheta}{{\boldsymbol \theta}}

\newcommand{\pil}{\pi^{(\ell)}}
\newcommand{\piL}{\pi^{(L)}}

\newcommand{\cml}{c_{\mathrm{ML}}}
\newcommand{\csl}{c_{\mathrm{SL}}}

\usepackage{mathtools}

\DeclarePairedDelimiter{\norm}{\lVert}{\rVert}
\DeclarePairedDelimiter{\parr}{(}{)}

\DeclareMathOperator*{\st}{\,:\,}

\newenvironment{keywords}
   {\begin{trivlist}\item[]{\bfseries\sffamily Keywords:}\ }
   {\end{trivlist}}

\ihead[]{}
\ohead[]{}
\ifoot[]{}
\ofoot[]{}

\newtheorem{lemma}{Lemma}
\newtheorem{theorem}{Theorem}

\newtheorem{assumption}{Assumption}
\newtheorem{proposition}{Proposition}
\newtheorem{remark}{Remark}

\ifpdf
\hypersetup{
  pdftitle={Multilevel Stein variational gradient descent with applications to Bayesian inverse problems},
  pdfauthor={Terrence Alsup, Luca Venturi, and Benjamin Peherstorfer}
}
\fi

\title{Multilevel Stein variational gradient descent with applications to Bayesian inverse problems}
\usepackage{times}

\author{Terrence Alsup\thanks{Courant Institute of Mathematical Sciences, New York University  (\texttt{alsup@cims.nyu.edu}, \texttt{venturi@cims.nyu.edu}, \texttt{pehersto@cims.nyu.edu})}
\and Luca Venturi\footnotemark[1] \and Benjamin Peherstorfer\footnotemark[1]}

\date{April 2021}

\begin{document}

\maketitle

\begin{abstract}
This work presents a multilevel variant of Stein variational gradient descent to more efficiently sample from target distributions. The key ingredient is a sequence of distributions with growing fidelity and costs that converges to the target distribution of interest. For example, such a sequence of distributions is given by a hierarchy of ever finer discretization levels of the forward model in Bayesian inverse problems. The proposed multilevel Stein variational gradient descent moves most of the iterations to lower, cheaper levels with the aim of requiring only a few iterations on the higher, more expensive levels when compared to the traditional, single-level Stein variational gradient descent variant that uses the highest-level distribution only. Under certain assumptions, in the mean-field limit, the error of the proposed multilevel Stein method decays by a log factor faster than the error of the single-level counterpart with respect to computational costs. Numerical experiments with Bayesian inverse problems show speedups of more than one order of magnitude of the proposed multilevel Stein method compared to the single-level variant that uses the highest level only.
\end{abstract}


\section{Introduction}
\label{sec:Intro}
Sampling from a target distribution $\pi$ is a common task in Bayesian inference. Typically, in machine learning, the (unnormalized) density of the target distribution can be evaluated  to approximately sample from it with Monte Carlo, variational, and particle methods \citep{RobertBook,pmlr-v33-ranganath14,pmlr-v37-rezende15,8588399}.  
We look at a setup that is more common in scientific machine learning and scientific computing, where a sequence of distributions $(\pil)$ is given that converges weakly to a computationally intractable target $\pi$ for increasing level $\ell \to \infty$. Here, intractable means that one cannot numerically evaluate the  (unnormalized) density of $\pi$.  
For example, one finds such a setup in Bayesian inverse problems \citep{S,KS,doi:10.1137/110845598}, where the target $\pi$ corresponds to a posterior distribution that depends on a forward model through the likelihood.
The forward model is typically a system of partial differential equations (PDEs) 
for which only numerical solutions can be computed; increasingly more accurate, more expensive discretizations (e.g., mesh width going to 0) of the forward-model PDEs then give rise to a sequence of distributions $(\pil)$ that converges to $\pi$.

To approximately sample from the target $\pi$, one then selects a level $L$ such that $\piL$ is a sufficiently accurate approximation of $\pi$ and then applies Monte Carlo or particle methods to $\piL$; see, e.g., \citep{S,KS,doi:10.1137/110845598}. 
One challenge of such an approach is that the density of $\piL$ can be computationally expensive to evaluate, because each evaluation of the density entails at least one numerical solve of the PDEs underlying the forward model, which can quickly make sampling from it prohibitively expensive.

\paragraph{Our contributions}
We propose to extend Stein variational gradient descent (SVGD) \citep{NIPS2016_b3ba8f1b} to a multilevel SVGD (MLSVGD) that  leverages the distributions from all levels $\ell = 1, \dots, L$ to more efficiently approximately sample from $\pi$ than traditional, single-level SVGD that uses the distribution $\piL$ on the highest level $L$ only; thus, the proposed MLSVGD builds on the long history of exploiting hierarchies of discretizations in scientific computing (see below for literature review). 
Our contributions are as follows: \textbf{(1)} an analysis that shows the cost complexity of the proposed MLSVGD is lower than the cost complexity of single-level SVGD; \textbf{(2)} a numerical algorithm that builds on an adaptive stopping criterion that can be applied in a black-box way; \textbf{(3)} numerical experiments with Bayesian inverse problems involving nonlinear diffusion-reaction and Euler-Bernoulli beam models that demonstrate that taking into account all levels $\ell = 1, \dots, L$ can lead to more than one order of magnitude  speedup compared to single-level SVGD.

\paragraph{Related work on multilevel methods in scientific computing}
Taking into account various discretizations and approximations of forward models to achieve computational speedups has a long tradition in scientific computing, e.g., multigrid solvers \citep{HackbushMG,doi:10.1137/1.9780898719505}, sparse grid approximations \citep{bungartz_griebel_2004}, multilevel Monte Carlo for estimating statistics \citep{10.1007/3-540-45346-6_5,doi:10.1287/opre.1070.0496,cliffe_multilevel_2011}; and multifidelity methods that leverage low-fidelity models without clear hierarchies \citep{doi:10.1137/16M1082469}. In terms of sampling from distributions, there is work on Markov chain Monte Carlo (MCMC) methods that exploit hierarchies of distributions such as 
multistage MCMC methods \citep{christen_markov_2005,Fox97samplingconductivity},
multilevel Metropolis--Hastings  \citep{doi:10.1137/130915005}; and MCMC methods with importance sampling \citep{Hoang_2013}. 
Then, there are multilevel/multifidelity variational methods, where a transport map (flow) is parametrized a priori; for example, \citep{AP20Context,PM18MultiTM} build on \citep{ELMOSELHY20127815,ParnoMCMCPrecon} and construct the transport maps from a distribution on a lower level and then use it as proposal for Metropolis-Hastings or for importance sampling. 
There are multilevel particle filters  \citep{doi:10.1137/17M1111553} and multilevel sequential Monte Carlo \citep{BESKOS20171417} methods, ensemble Kalman filtering \citep{doi:10.1137/15M100955X},
and extensions to nonlinear filtering using transport \citep{doi:10.1137/15M1038232}; these rely on telescoping sums of correlated differences between successive levels, whereas our approach uses the successive levels as preconditioners for sampling. Probably closest in style to our approach are the multilevel sequential Monte Carlo method \citep{LATZ2018154,doi:10.1137/19M1289601} and the multilevel cross-entropy method \citep{PKW17MFCE} that use distributions obtained on lower levels as starting distributions on higher levels.

\paragraph{Related work on SVGD from machine learning} 

The MLSVGD proposed in this work builds on SVGD introduced by \cite{NIPS2016_b3ba8f1b} and further theoretically analyzed in \citep{NIPS2017_17ed8abe}; extended to consider Newton directions \citep{NEURIPS2018_fdaa09fc}; exploiting geometry \citep{NEURIPS2019_eea5d933}, and other acceleration techniques \citep{pmlr-v97-liu19i}. 
A key building block for us will be recent advances on understanding the convergence properties of SVGD in the infinite particle (mean-field) regime. The work \citep{NIPS2017_17ed8abe,duncan2019geometry} shows the mean-field limit. The work \citep{NonAsym} shows non-asymptotic results. Further, the work \citep{chisqflow}  establishes exponential convergence under certain situations in the mean-field limit that motivates some of our assumptions. 
Another key building block is relating discretization error of the forward model at level $\ell$ to divergence of the corresponding posterior distributions $\pil$ with respect to the intractable target $\pi$, where we build on results by \cite{S} and an inequality involving the Kullback-Leibler (KL) divergence introduced by \cite{MX}.

\section{Preliminaries: Approximating measures with SVGD}

Let $\Theta \subset \R^d$ and $(\Theta, \mathcal{B}(\Theta))$ be a measurable space with $\mathcal{B}(\Theta)$ denoting the Borel $\sigma$-algebra of $\Theta$.  Consider approximating some target measure $\eta$ on $\Theta$ via an empirical measure, i.e., an ensemble of samples (particles); in the following, the distribution $\eta$, and all other distributions that will be considered, admit a density with respect to the Lebesgue measure over $\Theta$.  Moreover, the target distribution has the form $\eta \propto e^{-V}$ with the potential $V$.

\subsection{Approximating measures with SVGD}
\label{sec:SLSVGD:SVGD}

The SVGD method  \citep{NIPS2016_b3ba8f1b} iteratively moves forward an empirical distribution given by an ensemble $\{\btheta^{[i]}_t\}_{i = 1}^N$ from time $t$ to time $t + \delta$ via a map $\bphi_t(\btheta) = \btheta - \delta \bg_t(\btheta)$, where $\delta$ is a step size and $\bg_t : \Theta \to \R^d$ is a vector field.  SVGD chooses $\bg_t$ from a vector-valued reproducing kernel Hilbert space (RKHS) $\mathcal{H}^d$ with kernel $K: \R^d \times \R^d \to \R$ via a functional gradient descent step on the KL divergence (cf.~\eqref{eq:kldiv} in Appendix~\ref{appx:Metrics}). Denote the distribution of the particles at time $t$ as $\mu_t$ and define the functional $J_t(\bg) = \KL((I - \bg)_{\#} \mu_t  || \eta)$, where $(I - \bg)_{\#}\mu_t$ denotes the pushfoward measure. Then, SVGD chooses the gradient by setting $\bg_t = \nabla J_t(\boldsymbol 0)$, where $\boldsymbol 0$ is the zero function. Using the RKHS formulation, there is a closed form expression for $\nabla J_t({\boldsymbol 0})$, so that during the gradient descent the particles evolve according to the ordinary differential equation (ODE)
\begin{equation}
    \dot{\btheta}^{[i]}_t 
    = -\nabla J_t({\boldsymbol 0}) \left(\btheta^{[i]}_t \right) 
    = \bE_{\btheta^{\prime} \sim \mu_t}\left[ K(\btheta^{\prime}, \btheta^{[i]}_t)\nabla \log \eta(\btheta^{\prime}) + \nabla_1 K(\btheta^{\prime}, \btheta^{[i]}_t) \right] \, ,
\label{eq:Prelim:GradDescent}
\end{equation}
where $\nabla_1$ denotes the gradient with respect to the first argument.  In practice, the expectation is approximated using the empirical distribution of the ensemble of particles $\{\btheta^{[i]}_t\}_{i=1}^N$ and the ODE is integrated using the forward Euler method.  Thus, the SVGD update becomes
\begin{equation}
    \btheta^{[i]}_{t+\delta} 
    = \btheta^{[i]}_{t} + \frac{\delta}{N}\left( \sum\nolimits_{j = 1}^N \nabla_1 K(\btheta^{[j]}_t, \btheta_t^{[i]}) + \sum\nolimits_{j = 1}^N K(\btheta^{[j]}_t, \btheta^{[i]}_t) \nabla \log \eta (\btheta^{[j]}_t)  \right) \, .
\label{eq:Prelim:SVGDUpdate}
\end{equation}
In \citep{NIPS2017_17ed8abe}, the distribution of the particles $\{\btheta^{[i]}_t\}_{i=1}^N$ in the limit as $N \to \infty$ is given by the mean-field PDE
\begin{equation}
\partial_t \mu_t(\btheta) = - \nabla \cdot \left(\mu_t(\btheta) \bE_{\btheta^{\prime} \sim \mu_t}\left[ K(\btheta^{\prime}, \btheta)\nabla \log \eta(\btheta^{\prime}) + \nabla_1 K(\btheta^{\prime}, \btheta) \right] \right)\,,
\label{eq:Prelim:ContSVGD}
\end{equation}
with an initial measure $\mu_0$; see also \citep{chisqflow,SVAIS}.  \cite{NIPS2017_17ed8abe} shows that a steady state is reached in the limit $t\to \infty$ and the empirical distribution converges weakly (i.e. in distribution) to the target $\eta$.

\subsection{Approximating intractable target measures with SVGD}
\label{sec:SLSVGD:SVGD_intractable}

Consider now an intractable target distribution $\pi$; in contrast to the measure $\eta$ in Section~\ref{sec:SLSVGD:SVGD}, we can neither evaluate the (unnormalized) density of $\pi$ nor sample from $\pi$ directly.  Thus, the SVGD algorithm cannot directly be applied to $\pi$.  Instead, suppose we have a sequence of distributions $(\pil)_{\ell \ge 1}$ that converges weakly to $\pi$ for $\ell \to \infty$ (note that weak convergence is implied by convergence in the KL divergence) and call $\ell$ the level.  Moreover, we can evaluate the unnormalized density of each $\pil$ with computational costs $c_{\ell}$. Such a setup is common in Bayesian inverse problems; cf.~Section~\ref{sec:Intro}. 

\paragraph{Single-level approximation with SVGD} The aim is deriving a distribution $\mu$ that approximates $\pi$ with accuracy $\epsilon$. To quantify how close the approximation $\mu$ is to the target distribution, we consider the Hellinger distance $\dhell(\cdot, \cdot)$ in the following: 
First, select a level $L \in \mathbb{N}$ such that $\dhell(\piL, \pi) \leq \epsilon/2$.  Then, to approximate $\piL$ with SVGD, derive $\mu$ with accuracy $\dhell(\mu, \piL) \leq \epsilon/2$ from an initial distribution $\mu_0$; the triangle inequality leads to $\dhell(\mu, \pi) \leq \epsilon$. The fact that the Hellinger distance is a metric is important because it  allows us to separate the error due to truncating at level $L$ and the error due to the SVGD approximation of $\piL$; see Appendix~\ref{appx:Metrics} for the definition of the Hellinger distance. 

\paragraph{Computational costs} The costs of such an approach depend on two factors: (1) the costs $c_L$ of evaluating the density $\piL$ on level $L$, which is independent of SVGD, and (2) the costs of SVGD to find $\mu$ from $\piL$ with initial distribution $\mu_0$ to achieve $\dhell(\mu, \piL) \le \epsilon/2$. In the continuous SVGD given by Equation~\eqref{eq:Prelim:GradDescent}, we identify the costs of the approximation $\mu_T$ after integrating up to end time $T$ as 
\begin{equation}
\csl(T) = c_L T\, .
\label{eq:Prelim:CSL}
\end{equation}
We will see that the integration time $T$ depends on the divergence between $\mu_0$ and $\piL$. For the discrete SVGD given by Equation~\eqref{eq:Prelim:SVGDUpdate}, time is replaced with number of iterations and the costs must be multiplied by the number of particles $N$.  

\begin{remark} Although we use the Hellinger distance $\dhell$ in the following, the proposed analysis is also applicable if a different metric is used as long as it can be upper bounded by the KL divergence; see Section~\ref{sec:MLSVGD:SingleLevel} for more details.  Indeed, we make frequent use of the fact that the Hellinger distance can be bounded as
\begin{equation}
    2 \dhell(\rho_1, \rho_2)^2 \le \KL(\rho_1\ ||\ \rho_2) \,
    \label{eq:Prelim:div_rel}
\end{equation}
for two distributions $\rho_1,\rho_2$; see Lemma 2.4 of~\citep{Tsybakov} (note that the definition of Hellinger distance there is scaled by a constant factor $\sqrt{2}$).  The Hellinger distance is also useful because it can be used to bound the bias of a Monte Carlo estimator as shown in~\citep{S}.
\end{remark}

\section{A continuous multilevel Stein variational method and its cost complexity}
\label{sec:MLSVGD}

We propose MLSVGD that leverages the measures $\pi^{(1)}, \dots, \pi^{(L - 1)}$ with the aim to reduce the costs of approximating $\piL$ compared to the traditional, single-level SVGD that uses $\piL$ only.  Our analysis of the proposed MLSVGD method is conducted in the time-continuous and mean-field setting where the SVGD measures satisfy the PDE~\eqref{eq:Prelim:ContSVGD} and the particles satisfy the ODE~\eqref{eq:Prelim:GradDescent}.  A discrete, heuristic, algorithmic formulation follows in Section~\ref{sec:Algorithm} with a numerical comparison to single-level SVGD in Section~\ref{sec:NumericalResults}.

\begin{figure}
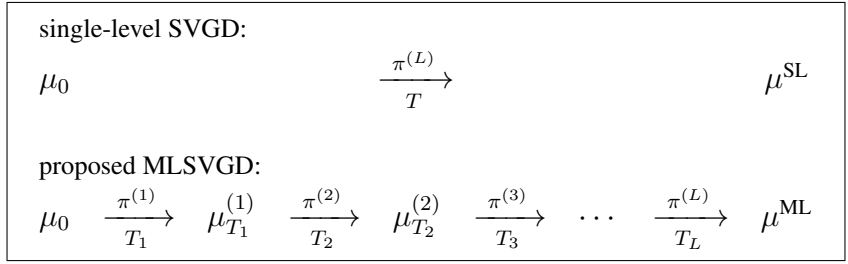

\large
\begin{center}
\fbox{
\begin{tabular}{ccccccccc}
\multicolumn{8}{l}{\normalsize single-level SVGD:}\\
$\mu_0$ & \multicolumn{7}{c}{$\xrightarrow[T]{~~~~~~~~~~~~~~~~~~~~~~~~~~ ~~~~~~~~~~~~~~~~~~~~~~\piL ~~~~~~~~~~~~~~~~~~~~~~~~~~~~~~~~~~~~~~~~~~~~~~~~}$} & $\mu^{\text{SL}}$\\
&&&&&&&&\\
\multicolumn{8}{l}{\normalsize proposed MLSVGD:}\\
$\mu_0$ & $\xrightarrow[T_1]{~~~~\pi^{(1)}~~~~}$ & $\mu_{T_1}^{(1)}$ & $\xrightarrow[T_2]{~~~~\pi^{(2)}~~~~}$ & $\mu_{T_2}^{(2)}$ & $\xrightarrow[T_3]{~~~~\pi^{(3)}~~~~}$ & $\cdots$ & $\xrightarrow[T_L]{~~~~\piL~~~~}$ & $\mu^{\text{ML}}$
\end{tabular}
}
\end{center}
\caption{The proposed MLSVGD leverages a hierarchy of distributions with increasing costs and fidelity with the aim of requiring fewer iterations on the higher, more expensive levels compared to traditional, single-level SVGD that uses the highest-level distribution only.}
\label{fig:MLSVGD}
\end{figure}

\subsection{Continuous MLSVGD}
\label{sec:MLSVGD:MLSVGD}

To describe the proposed MLSVGD, consider the levels $\ell = 1, \dots, L$ and let $\mu_0$ be an initial distribution.  At level $\ell = 1$, we define $\mu_{T_1}^{(1)}$ as the distribution of the continuous SVGD \eqref{eq:Prelim:ContSVGD} at time $T_1$  with target $\pi^{(1)}$ and initial $\mu_0$.  At level $\ell = 2$, we obtain $\mu_{T_2}^{(2)}$ at time $T_2$ with the target $\pi^{(2)}$ and initial distribution $\mu_{T_1}^{(1)}$.  In general, at level $\ell$, we obtain $\mu^{(\ell)}_{T_{\ell}}$ at time $T_{\ell}$ with target $\pil$ and initial distribution $\mu_{T_{\ell - 1}}^{(\ell - 1)}$.  Thus, deriving $\mu^{(\ell)}_{T_{\ell}}$ is an iterative process over the levels $1, \dots, \ell - 1$, depicted in Figure~\ref{fig:MLSVGD}, of first computing $\mu_{T_1}^{(1)}, \dots, \mu_{T_{\ell - 1}}^{(\ell - 1)}$.  The costs of MLSVGD are given by
\begin{equation}
    \cml(T_1, \dots, T_L) = \sum_{\ell = 1}^L c_{\ell} T_{\ell}\,,
\label{eq:MLSVGD:MLSVGD_costs}
\end{equation}
cf.~the costs $\csl(T) = c_L T$ of the single-level SVGD as defined in \eqref{eq:Prelim:CSL}.

\subsection{Assumptions for cost complexity analysis of single-level SVGD and MLSVGD}

We build on the following three assumptions to derive the cost complexity of both traditional single-level SVGD as well as the proposed MLSVGD.  The first assumption is a standard assumption in scientific computing on the cost of evaluating the densities, while the second and third are needed to certify that $\dhell(\mu, \pi) \le \epsilon$.

\begin{assumption}
The costs $c_{\ell}$ of evaluating the (unnormalized) density $\pi^{(\ell)}$ are bounded as
\[
c_{\ell} \leq c_0 s^{\gamma \ell}\,,\qquad \ell \in \mathbb{N}\,,
\]
with constants $c_0, \gamma > 0$ independent of $\ell$ and $s > 1$.
\label{asm:Costs}
\end{assumption}

\begin{assumption}
There exists $\alpha, k_0, k_1 > 0$ independent of $\ell$ such that $\KL(\mu_0 || \pi^{(\ell)}) \leq k_0$ for all $\ell \in \mathbb{N}$ and 
\[
\KL(\pi^{(\ell)} || \pi) \leq k_1 s^{-\alpha \ell}\,,\qquad \ell \in \mathbb{N}\,,
\]
where $s$ is the same constant independent of $\ell$ as in Assumption~\ref{asm:Costs} and $\mu_0$ is the initial distribution.
\label{asm:BiasRate}
\end{assumption}

\begin{assumption}
There exists a rate $\lambda > 0$ such that for any initial distribution $\nu_0$
\[
\KL(\nu_t || \pil) \leq \mathrm e^{-\lambda t} \KL(\nu_0 || \pil)\,,\qquad \ell \in \mathbb{N}\,,
\]
holds, where $\nu_t$ solves the mean-field SVGD equation \eqref{eq:Prelim:ContSVGD} at time $t$.
\label{asm:SVGDRate}
\end{assumption}
\cite{NonAsym} show that Assumption~\ref{asm:SVGDRate} is satisfied if the measures $\pil$ satisfy a Stein log-Sobolev inequality.  \cite{chisqflow} also show that Assumption~\ref{asm:SVGDRate} is satisfied for a specific choice of the kernel $K$. We also note that the exponential convergence rate for the KL divergence appears in the theory for the convergence of Markov processes when the target measure satisfies a log-Sobolev inequality \citep[Theorem~5.2.1]{markovdiff}; however, SVGD approximates the gradient in an RKHS and thus \citep[Theorem~5.2.1]{markovdiff} is not directly applicable.

\subsection{Cost complexity of continuous single-level SVGD}
\label{sec:MLSVGD:SingleLevel}

Consider the single-level SVGD that selects $L$ such that $\dhell(\piL, \pi) \leq \epsilon/2$ and then starts with a $\mu_0$ to find $\mu^{\text{SL}}_T$ that satisfies $\dhell(\mu_T^{\text{SL}}, \piL) \leq \epsilon/2$. For brevity, we write $\mu^{\text{SL}} = \mu^{\text{SL}}_T$.  The following proposition bounds the costs of this single-level SVGD with respect to the tolerance $\epsilon$.

\begin{proposition}
If Assumptions~\ref{asm:Costs}--\ref{asm:SVGDRate} hold, then the costs of continuous single-level SVGD to obtain $\mu^{\text{SL}}$ with
\[
\dhell(\mu^{\mathrm{SL}}, \pi) \leq \epsilon
\]
is bounded as
\begin{equation}
\csl^*(\epsilon) \le \frac{2c_0 s^{\gamma}}{\lambda} \left( \frac{\sqrt{2k_1}}{\epsilon}  \right)^{2\gamma/\alpha} \log \left( \frac{\sqrt{\KL(\mu_0\ ||\ \piL)}}{\sqrt{2}\epsilon} \right) \,.
\label{eq:MLSVGD:SLSVGDCost}
\end{equation}
\label{prop:hificostcomp}
\end{proposition}

\begin{proof}
By the triangle inequality for the Hellinger distance we have that
\[
    \dhell(\mu^{\text{SL}}, \pi) \le 
    \dhell(\mu^{\text{SL}}, \piL) 
    + \dhell(\piL, \pi) ,
\]
so we will bound both of these terms independently by $\epsilon/2$.  By inequality~\eqref{eq:Prelim:div_rel}, it is sufficient to   bound the KL divergence because
\begin{equation}
    \dhell(\mu^{\text{SL}}, \piL) \le \sqrt{ \frac{\KL(\mu^{\text{SL}}\ ||\ \piL )}{2} },
\end{equation}
and similarly for $\dhell(\piL, \pi)$.  By Assumption~\ref{asm:BiasRate} choose $L$ to be
\begin{equation}
    L = \left\lceil \frac{1}{\alpha} \log_s \left( \frac{2 k_1}{\epsilon^2} \right) \right \rceil \le  \frac{1}{\alpha} \log_s \left( \frac{2 k_1}{\epsilon^2} \right) + 1,
    \label{eq:highfidlevel}
\end{equation}
so that 
\begin{equation}
    \dhell(\piL, \pi) \le \sqrt{ \frac{\KL(\piL\ ||\ \pi)}{2}} \le \sqrt{ \frac{k_1 s^{-\alpha L}}{2} } \le \frac{\epsilon}{2} .
\end{equation}
Now by Assumptions~\ref{asm:SVGDRate} the time needed to integrate with SVGD to achieve $\dhell(\mu^{\text{SL}}, \piL) \leq \epsilon/2$ is
\begin{equation}
    T^*_{\text{SL}} \le \frac{1}{\lambda } \log\left( \frac{\text{KL}(\mu_0\ ||\ \piL)}{2\epsilon^2}  \right) .
\end{equation}
The total cost to integrate until time $T^*_{\text{SL}}$ at level $L$ is thus
\[
\csl^*(\epsilon) = c_0 s^{\gamma L} T_{\mathrm{SL}}^*  \le \frac{2c_0 s^{\gamma}}{\lambda} \left( \frac{\sqrt{2k_1}}{\epsilon}  \right)^{2\gamma/\alpha} \log \left( \frac{\sqrt{\text{KL}(\mu_0\ ||\ \piL)}}{\sqrt{2}\epsilon} \right)\,.
\]
\end{proof}

\paragraph{Discussion of cost complexity of single-level SVGD}
The bound \eqref{eq:MLSVGD:SLSVGDCost} in Proposition~\ref{prop:hificostcomp} shows that if we start with an initial distribution $\mu_0$ that has a large $\KL$ divergence $\KL(\mu_0|| \piL)$ with respect to $\piL$, then we will need to integrate for a long time with SVGD to reach our tolerance. The proposed MLSVGD is aiming to avoid the long time integration by starting the integration at the highest level $L$ with good initial distributions found on the cheaper, lower levels $\ell = 1, \dots, L - 1$ that are closer to $\piL$ in the KL divergence than $\mu_0$.

\subsection{Cost complexity of continuous MLSVGD}

Consider now the MLSVGD approach of Section~\ref{sec:MLSVGD:MLSVGD}.  We need to make one additional assumption compared to the single-level SVGD regarding the KL divergence between consecutive measures $\pil$ and $\pi^{(\ell-1)}$ that will allow us to chain them together as in Figure~\ref{fig:MLSVGD}. 
\begin{assumption}
There exists a constant $k_2 > 0$ independent of $\ell$ such that $\KL(\pi^{(\ell - 1)} || \pil) \leq k_2 s^{-\alpha \ell}$\,, where $\alpha$ is the same rate as in Assumption~\ref{asm:BiasRate}.
\label{asm:MLBiasRate}
\end{assumption}
The key result is to use a triangle-\emph{like} inequality as in Appendix~\ref{appdx:kltriangle} to decompose the KL divergence.  In particular,
\begin{equation}
    \KL\left(\mu^{(\ell-1)}_{T_{\ell-1}}\ ||\ \pil\right) = \KL\left(\mu^{(\ell-1)}_{T_{\ell-1}}\ ||\ \pi^{(\ell - 1)}\right) 
    + \KL\left(\pi^{(\ell-1)}\ ||\ \pil\right) + R_{\ell}
    \label{eq:kltriangle}
\end{equation}
with the remainder $R_{\ell}$ given by
\begin{equation}
    R_{\ell} = \int_{\R^d} \left(\mu^{(\ell-1)}_{T_{\ell-1}}(\btheta) - \pi^{(\ell-1)}(\btheta) \right) \log\left( \frac{\pi^{(\ell-1)}(\btheta)}{\pil(\btheta)}  \right)\ \mathrm{d}\btheta  .
    \label{eq:ConstCL}
\end{equation}
Because $\pil$ converges to $\pi$, we have that $\pi^{(\ell-1)}/\pil \to 1$ pointwise and hence $\log\left( \pi^{(\ell-1)}/{\pil}\right) \to 0$.  Moreover, $\mu^{(\ell)}_{T_{\ell}} \to \pi^{(\ell)}$ as $T_{\ell} \to \infty$.  Thus, $R_{\ell} \to 0$.  In particular $R_{\ell}$ is a bounded sequence meaning that there is some constant $R \ge R_{\ell}$ for all $\ell$. 
The following proposition give bounds on the costs of MLSVGD. The later Proposition~\ref{prop:MLSVGD:CostsWithConvergentC} will give a faster decaying bound on the costs if $R_{\ell}$ goes to zero with a known rate, as in our Bayesian inverse problems in Section~\ref{sec:BayesianInverseProblems}.

\begin{proposition}
If Assumptions~\ref{asm:Costs}--\ref{asm:MLBiasRate} hold, then continuous MLSVGD gives $\mu^{\text{ML}}$ with $\dhell(\mu^{\text{ML}}, \pi) \leq \epsilon$ with costs bounded as
\[
\cml^*(\epsilon) \le \frac{2 c_0 s^{2\gamma}}{\lambda \gamma \log(s)} \left( \frac{\sqrt{2k_1}}{\epsilon}  \right)^{2\gamma/\alpha} \log\left( \frac{\sqrt{\epsilon^2 + 2(k_2 + R )}}{\epsilon} \right) \,,
\]
where $R$ bounds \eqref{eq:ConstCL}.
\label{prop:mlcostcomp}
\end{proposition}

\begin{proof}
As in Equation~\eqref{eq:highfidlevel} in the proof of Proposition~\ref{prop:hificostcomp} we select the level $L$ as
\begin{equation}
    L = \left\lceil \frac{1}{\alpha} \log_s \left( \frac{2 k_1}{\epsilon^2} \right) \right\rceil \le \frac{1}{\alpha} \log_s \left( \frac{2 k_1}{\epsilon^2} \right) + 1,
\end{equation}
so that $\dhell(\piL, \pi) \le \epsilon/2$.  The total cost for the continuous MLSVGD is
\begin{equation}
    \cml^*(\epsilon) = \sum_{\ell=1}^L c_0 s^{\gamma \ell} T_{\ell}\, ,
\end{equation}
where it remains to choose the integration times $T_{\ell}$ at each level.  To do this we balance the KL divergence of the SVGD approximation with the KL divergence due to the fidelity.  By Equation~\eqref{eq:kltriangle}, we have
\begin{equation}
\KL(\mu^{(\ell)}_{T_{\ell}}\ ||\ \pi^{(\ell)}) \le e^{-\lambda  T_{\ell} }
\left( \KL\left(\mu^{(\ell-1)}_{T_{\ell-1}}\ ||\ \pi^{(\ell-1)} \right) 
+  \KL\left(\pi^{(\ell - 1)}\ ||\ \pi^{(\ell)} \right) + R_{\ell}  \right)\, ,
\label{eq:recursive}
\end{equation}
giving a recursive bound on the KL divergence in terms of the KL divergence at the previous level.  At each level $\ell$ choose the integration time $T_{\ell}$ so that 
\begin{equation}
    \KL\left(\mu^{(\ell)}_{T_{\ell}}\ ||\ \pi^{(\ell)}\right) \le \frac{\epsilon^2}{2} 
    \label{eq:epsiloneq}
\end{equation}
is satisfied.  In particular, at the final level $L$ we will have that $\KL(\mu^{\text{ML}}\ ||\ \piL) \le \epsilon^2/2$ and hence $\dhell(\mu^{\text{ML}}\ ||\ \piL) \le \epsilon/2$ as desired.  By choosing $T_{\ell}$ so that this is satisfied at every level we have from Equation~\eqref{eq:recursive} that
\begin{equation}
    \KL\left(\mu^{(\ell)}_{T_{
    \ell}}\ ||\ \pi^{(\ell)}\right) 
    \le e^{-\lambda  T_{\ell} }\left( \frac{\epsilon^2}{2} +  \KL(\pi^{(\ell - 1)}\ ||\ \pi^{(\ell)} ) + R_{\ell}  \right) \le \frac{\epsilon^2}{2}\, .
\end{equation}
Thus, we choose $T_{\ell}$ sequentially so that Equation~\eqref{eq:epsiloneq} is always satisfied.  As a result, the integration time $T_{\ell}$ needed at each level $\ell$ is bounded by
\begin{equation}
    T_{\ell} \le \frac{1}{\lambda}\log\left( 1 + \frac{2(\KL(\pi^{(\ell - 1)}\ ||\ \pi^{(\ell)}) + R_{\ell} )}{\epsilon^2} \right)\, .
\end{equation}
Finally, the total cost can be bounded by
\begin{equation}
    \cml^*(\epsilon) \le \sum_{\ell=1}^L \frac{c_0}{\lambda} s^{\gamma \ell} \log\left( 1 + \frac{2(\KL(\pi^{(\ell - 1)}\ ||\ \pi^{(\ell)}) + R_{\ell} )}{\epsilon^2} \right)\, .
\end{equation}
We now use the fact that $R_{\ell} \le R$ and $\KL(\pi^{(\ell - 1)}\ ||\ \pi^{(\ell)}) \le k_2$ to obtain 
\begin{equation}
    \cml^*(\epsilon) \le \sum_{\ell=1}^L \frac{c_0}{\lambda} s^{\gamma \ell} \log\left( 1 + \frac{2(k_2 + R )}{\epsilon^2} \right)\, .
\end{equation}
Since the terms in this sum are increasing, we can upper bound the cost further by switching to an integral
\begin{equation}
\begin{split}
    \cml^*(\epsilon) &\le \int_0^{L+1}  \frac{c_0}{\lambda} s^{\gamma x} \log\left( 1 + \frac{2(k_2 + R )}{\epsilon^2} \right)\ \mathrm{d}x\\
    &= \frac{c_0}{\lambda \gamma \log(s)} \log\left( 1 + \frac{2(k_2 + R )}{\epsilon^2} \right) s^{\gamma(L+1)} \\
    &\le \frac{2 c_0 s^{2\gamma}}{\lambda \gamma \log(s)} \left( \frac{\sqrt{2k_1}}{\epsilon}  \right)^{2\gamma/\alpha} \log\left( \frac{\sqrt{\epsilon^2 + 2(k_2 + R )}}{\epsilon} \right)\, .
\end{split}
\end{equation}
\end{proof}

We now consider the case where the remainder term behaves as $R_{\ell} \lesssim s^{-\alpha \ell}$, which allows us to make a more efficient choice when selecting the integration time $T_{\ell}$ at each level. In particular, it allows us to set $T_{\ell}$ such that 
\begin{equation}
	\KL(\mu^{(\ell)}_{T_{\ell}}\ ||\ \pi^{(\ell)}) \sim s^{-\alpha \ell} \, ,
\end{equation}
which leads to the following proposition that shows an improved cost complexity compared to Proposition~\ref{prop:mlcostcomp}. 

\begin{proposition}
If Assumptions~\ref{asm:Costs}--\ref{asm:MLBiasRate} hold and $R_{\ell} \leq k_3 s^{-\alpha \ell}$, then the costs of continuous MLSVGD to have \\$\dhell(\mu^{\mathrm{ML}},\ \pi) \le \epsilon$ can be bounded as
\begin{equation}
    \cml^*(\epsilon) \le \frac{c_0 s^{2 \gamma}}{\lambda \gamma \log(s)} \log\left( s^{\alpha} + \frac{k_2 + k_3}{k_1} \right) \left( \frac{\sqrt{2k_1}}{\epsilon}  \right)^{2\gamma/\alpha}\, .
    \label{eq:MLSVGD:FastC}
\end{equation}
\label{prop:MLSVGD:CostsWithConvergentC}
\end{proposition}

\begin{proof}
Starting from Equation~\eqref{eq:recursive} in the proof of Proposition~\ref{prop:mlcostcomp} change $\epsilon$ to instead be
\begin{equation}
    \epsilon_{\ell} = \sqrt{2 k_1} s^{-\alpha \ell / 2}\, 
\end{equation}
at each level $\ell$.  By Assumption~\ref{asm:BiasRate} we know that $L$ is chosen so that
\begin{equation}
  \epsilon_L^2 =  2k_1 s^{-\alpha L} \le \epsilon^2\, ,
\end{equation}
so that $\epsilon_L \le \epsilon$.  Plugging in this choice gives that the integration times needed are
\begin{equation}
    T_{\ell}^* \le \frac{1}{\lambda} \log\left( s^{\alpha} + \frac{\KL(\pi^{(\ell-1)}\ ||\ \pi^{(\ell)}) + R_{\ell}}{k_1 s^{-\alpha \ell}} \right)\, .
\end{equation}
By Assumption~\ref{asm:MLBiasRate} and the assumption in the proposition, we have that
\begin{equation}
    T_{\ell}^* \le  \frac{1}{\lambda} \log\left( s^{\alpha} + \frac{k_2 + k_3}{k_1} \right)\, ,
\end{equation}
so that the integration time is fixed at each level.  The cost is now bounded by
\begin{equation}
    \cml^*(\epsilon) \le \sum_{\ell=1}^L \frac{c_0}{\lambda} s^{\gamma \ell} \log\left( s^{\alpha} + \frac{k_2 + k_3}{k_1} \right)\, .
\end{equation}
Since the terms in the sum are increasing, we can further bound this with an integral :
\begin{equation}
    \cml^*(\epsilon) \le \sum_{\ell=1}^L \frac{c_0}{\lambda} s^{\gamma \ell} \log\left( s^{\alpha} + \frac{k_2 + k_3}{k_1} \right) 
    \le \int_1^{L+1} \frac{c_0}{\lambda} s^{\gamma x} \log\left( s^{\alpha} + \frac{k_2 + k_3}{k_1} \right)\ \mathrm{d}x\, .
\end{equation}
Computing the integral gives
\begin{equation}
    \cml^*(\epsilon) \le \frac{c_0 s^{2\gamma}}{\lambda \gamma \log(s)} \log\left( s^{\alpha} + \frac{k_2 + k_3}{k_1} \right) s^{\gamma L}\, .
\end{equation}
Finally, by plugging in $L$ we obtain
\begin{equation}
    \cml^*(\epsilon) \le  \frac{c_0 s^{2\gamma}}{\lambda \gamma \log(s)} \log\left( s^{\alpha} + \frac{k_2 + k_3}{k_1} \right) \left( \frac{\sqrt{2k_1}}{\epsilon}  \right)^{2\gamma/\alpha}\, .
\end{equation}
\end{proof}

\paragraph{Discussion of cost complexity of MLSVGD} 
Looking at the single-level SVGD and MLSVGD cost bounds from Propositions~\ref{prop:hificostcomp} and~\ref{prop:MLSVGD:CostsWithConvergentC}, respectively, we note two major differences.  The first is that there is no $\log \epsilon^{-1}$ term in the cost bound \eqref{eq:MLSVGD:FastC} of Proposition~\ref{prop:MLSVGD:CostsWithConvergentC} and thus MLSVGD achieves a cost complexity that grows by $\log \epsilon^{-1}$ slower than the cost complexity of single-level SVGD as $\epsilon \to 0$.  Moreover, whenever $\epsilon \to 0$, we now have a fixed integration time at each level $\ell$ as opposed to requiring an increasing number of iterations as the level goes to infinity as in the single-level case.  The second notable difference is that the constant $k_0$, which depends on the KL divergence from the initial distribution $\mu_0$ and the target $\pi$, does not appear in~\eqref{eq:MLSVGD:FastC}.  Instead the bound~\eqref{eq:MLSVGD:FastC} depends on the constant $k_2$ from Assumption~\ref{asm:MLBiasRate}, which depends only on the KL divergence between two consecutive levels.  Thus, if the KL divergence between consecutive levels is low, then the previous level serves as a good preconditioner for the next level leading to reduced costs.

\begin{remark}
    The order $\log \epsilon^{-1}$ comes from the exponential decay rate of the KL divergence for SVGD in Assumption~\ref{asm:SVGDRate}. If the assumption is violated and, for example, the KL divergence decays only algebraically, then we expect the speedup to be on the order of $\epsilon^{-\beta}$ for some constant $\beta > 0$. This is further supported by our numerical results that indicate that MLSVGD obtains speedups even if SVGD converges slower than in Assumption~\ref{asm:SVGDRate}. We leave the detailed analysis of this to future work. 
    \label{rm:Rate}
\end{remark}

\section{MLSVGD for Bayesian inverse problems}
\label{sec:BayesianInverseProblems}

Typically, in Bayesian inverse problems in scientific computing, one is interested in inferring an unknown quantity $\btheta$ from some noisy observed data $\by = G(\btheta^*) + {\boldsymbol e}$ with $G$ denoting the parameter-to-observable map and ${\boldsymbol e}$ being the noise; see, e.g., \citep{S,KS,doi:10.1137/110845598}. Let $\pi_0$ be the prior and consider zero-mean Gaussian noise with covariance $\bGamma$, then the posterior is given by
\begin{equation}
    \pi( \btheta ) = \frac{1}{Z} \exp\left( -\frac{1}{2}\| \by - G(\btheta) \|_{\bGamma^{-1}}^2 \right) \pi_0( \btheta )
    \label{eq:posterior}\, ,
\end{equation}
with the normalizing constant 
\begin{equation}
    Z = \int_{\Theta} \exp\left( -\frac{1}{2}\| \by - G(\btheta)\|_{\bGamma^{-1}}^2 \right) \pi_0(\btheta) \ \mathrm{d}\btheta\,,
\label{eq:normalizing}
\end{equation}
where $\|\bu\|_{\bGamma^{-1}} = \langle \bGamma^{-1}\bu, \bu \rangle$.  Now let $(G_{\ell})_{\ell \ge 1}$ denote a sequence of approximations to the parameter-to-observable map $G$, e.g., given by finite-difference or finite-element discretizations of the PDEs underlying $G$, and define $\pi^{(\ell)}$ by replacing $G$ with $G_{\ell}$ and define $Z_{\ell}$ similarly.  The next two assumptions will be sufficient to apply our results from Section~\ref{sec:MLSVGD}.

\begin{assumption}[Model error]
There is a function $\psi : \mathbb{N} \to (0,\infty)$, with $\psi(\ell) \to 0$ as $\ell \to \infty$, such that
\begin{equation}
	\|G(\btheta) - G_{\ell}(\btheta)\|_{L^2(\pi_0)} \le \psi(\ell) \, ,
\label{eq:modelerror}
\end{equation}
where the $\|\cdot\|_{L^2(\pi_0)}$ is  the $L^2$ norm over $\pi_0$; cf.~\eqref{eq:appx:L2Norm} in Appendix~\ref{appx:Metrics}.
\label{asm:modelerror}
\end{assumption}

\begin{assumption}
There exists a constant $b_3 > 0$ independent of $\ell$ such that
\begin{equation}
    \mu^{(\ell)}_{T_{\ell}}(\btheta) \le b_3 \pi_0(\btheta)
\end{equation}
for all $\ell \ge 1$.
\label{asm:svgdpdfbdd}
\end{assumption}

The next theorem shows that if Assumptions~\ref{asm:Costs},~\ref{asm:SVGDRate},~\ref{asm:modelerror}, and~\ref{asm:svgdpdfbdd} are satisfied, then our cost complexity results derived for MLSVGD in Section~\ref{sec:MLSVGD} hold in the Bayesian inverse problem setting. These assumptions can be interpreted in the context of Bayesian inverse problems as follows: Assumption~\ref{asm:Costs} and Assumption~\ref{asm:modelerror} are related to the forward model. Together they state that the approximation $G_{\ell}$ converges in an $L^2$-sense to $G$ as the level $\ell$ is increased. At the same time, as the level $\ell$ is increased and $G_{\ell}$ gets closer to $G$, the computational costs of evaluating $G_{\ell}$ may increase with a rate $\gamma$. This is typical behavior in, e.g., finite-element forward models where refining the mesh (increasing the level) leads to more accurate approximations and at the same time the computational costs of computing the finite-element solution increase with the number of mesh points. Furthermore, Assumption~\ref{asm:modelerror} is similar to the assumptions of \citep[Corollary 4.9]{S}, although there a pointwise bound is used. Assumption~\ref{asm:SVGDRate} is the convergence rate of SVGD and motivated by results from the literature as discussed in Section~\ref{sec:MLSVGD}; cf.~Remark~\ref{rm:Rate} for other convergence behavior. Assumption~\ref{asm:svgdpdfbdd} ensures that the tail of the posterior distribution behaves as the tail of the prior and is similar to the envelope assumption made in, e.g., acceptance/rejection sampling \citep{RobertBook}.

\begin{theorem}
    If Assumptions~\ref{asm:Costs},~\ref{asm:SVGDRate}, and~\ref{asm:svgdpdfbdd} hold and Assumption~\ref{asm:modelerror} holds with $\psi(\ell) = b_0 s^{-\alpha \ell}$, then Assumptions~\ref{asm:BiasRate} and~\ref{asm:MLBiasRate} hold and thus the cost complexity to find $\mu^{\text{ML}}$ with $\dhell(\mu^{\text{ML}}, \pi) \le \epsilon$ is given by
    \begin{equation}
        \cml^*(\epsilon) \le \frac{c_0 s^{2 \gamma}}{\lambda \gamma \log(s)} \log\left( s^{\alpha} + (1+s^{\alpha})\left(\frac{4}{3} + \frac{b_3}{3b_1b_2} \right) \right) \left( \frac{\sqrt{3b_1b_2b_0}}{\epsilon}  \right)^{2\gamma/\alpha} \, ,
    \end{equation}
    where the constants $b_1,b_2$ are independent of $\epsilon$ and given in the proof of Lemma~\ref{lmma:klmodelbound} in Appendix~\ref{appdx:proof1}.
    \label{thm:BayesCostComp}
\end{theorem}

\begin{proof}
By Lemma~\ref{lmma:bipklrate} in Appendix~\ref{appdx:proof2} we know that Assumptions~\ref{asm:BiasRate} and~\ref{asm:MLBiasRate} hold with $k_1 = Cb_0$ and $k_2 = Cb_0(1 + s^{\alpha})$.  Thus, we just need to verify that $R_{\ell} \le k_3 s^{-\alpha \ell}$ for some constant $k_3$ to apply Proposition~\ref{prop:mlcostcomp}.
\begin{equation}
\begin{split}
    R_{\ell} &= \int_{\Theta} \left( \mu^{(\ell-1)}_{T_{\ell-1}}(\btheta) - \pi^{(\ell - 1)}(\btheta) \right) \log \left( \frac{\pi^{(\ell-1)}(\btheta)}{\pi^{(\ell)}(\btheta)} \right)\ d\btheta       \\
    &=  \int_{\Theta} \left( \mu^{(\ell-1)}_{T_{\ell-1}}(\btheta) - \pi^{(\ell - 1)}(\btheta) \right) \log \left(  \frac{Z_{\ell} \exp\left( -\frac{1}{2}\|\by - G_{\ell-1}(\btheta)\|_{\bGamma^{-1}}^2 \right)}{Z_{\ell - 1} \exp\left( -\frac{1}{2}\|\by - G_{\ell}(\btheta)\|_{\bGamma^{-1}}^2 \right) }  \right) \ \mathrm{d}\btheta \\
    &=  \int_{\Theta} \left( \mu^{(\ell-1)}_{T_{\ell-1}}(\btheta) - \pi^{(\ell - 1)}(\btheta) \right) \log \left(  \frac{ \exp\left( -\frac{1}{2} \|\by - G_{\ell-1}(\btheta) \|_{\bGamma^{-1}}^2 \right)}{ \exp\left( -\frac{1}{2}\|\by - G_{\ell}(\btheta)\|_{\bGamma^{-1}}^2 \right) }  \right) \ \mathrm{d}\btheta  ,
\end{split}
\end{equation}
where the last line follows from the fact that
\begin{equation}
     \int_{\Theta} \left( \mu^{(\ell-1)}_{T_{\ell-1}}(\btheta) - \pi^{(\ell - 1)}(\btheta) \right) \log \left(  \frac{Z_{\ell} }{Z_{\ell - 1}  }  \right) \ \mathrm{d}\btheta = 0
\end{equation}
since $\frac{Z_{\ell} }{Z_{\ell - 1}  }$ is a constant and $\pi^{(\ell - 1)}$ and $\mu^{(\ell-1)}_{T_{\ell-1}}$ both integrate to one.  By the triangle inequality we have that
\begin{equation}
\begin{split}
    R_{\ell} &\le \frac{1}{2} \int_{\Theta} \left| \| \by - G_{\ell}(\btheta)\|_{\bGamma^{-1}}^2 - \| \by - G_{\ell-1}(\btheta)\|_{\bGamma^{-1}}^2  \right|  \mu^{(\ell-1)}_{T_{\ell-1}}(\btheta)  \ \mathrm{d}\btheta  \\
    &\quad + \frac{1}{2} \int_{\Theta} \left| \| \by - G_{\ell}(\btheta)\|_{\bGamma^{-1}}^2 - \| \by - G_{\ell-1}(\btheta)\|_{\bGamma^{-1}}^2  \right|  \pi^{(\ell-1)}(\btheta)  \ \mathrm{d}\btheta\, .
\end{split}    
\end{equation}
We have that 
\begin{equation}
    \pi^{(\ell-1)}(\btheta) \le \frac{1}{Z_{\ell-1} }\pi_0(\btheta)\, ,
\end{equation}
so that when combined with Assumption~\ref{asm:svgdpdfbdd}
\begin{equation}
\begin{split}
    R_{\ell} &\le \frac{1}{2} \int_{\Theta} \left| \| \by - G_{\ell}(\btheta)\|_{\bGamma^{-1}}^2 - \| \by - G_{\ell-1}(\btheta)\|_{\bGamma^{-1}}^2  \right|  \mu^{(\ell-1)}_{T_{\ell-1}}(\btheta)  \ \mathrm{d}\btheta  \\
    &\quad + \frac{1}{2} \int_{\Theta} \left| \| \by - G_{\ell}(\btheta)\|_{\bGamma^{-1}}^2 - \| \by - G_{\ell-1}(\btheta)\|_{\bGamma^{-1}}^2  \right|  \pi^{(\ell-1)}(\btheta)  \ \mathrm{d}\btheta \\
    &\le \frac{b_3}{2} \int_{\Theta} \left| \| \by - G_{\ell}(\btheta)\|_{\bGamma^{-1}}^2 - \| \by - G_{\ell-1}(\btheta)\|_{\bGamma^{-1}}^2  \right|  \pi_{0}(\btheta)  \ \mathrm{d}\btheta  \\
    &\quad + \frac{1}{2 Z_{\ell - 1}} \int_{\Theta} 
    \left| \| \by - G_{\ell}(\btheta)\|_{\bGamma^{-1}}^2 - \| \by - G_{\ell-1}(\btheta)\|_{\bGamma^{-1}}^2  \right|  \pi_{0}(\btheta)  \ \mathrm{d}\btheta \\
    &\le \left(\frac{b_3}{2} + \frac{b_1 b_2}{2}\right) \|G_{\ell} - G_{\ell-1}\|_{L^2(\pi_0)} ,
\end{split}
\end{equation}
so that $k_3 = \left(\frac{b_3}{2} + \frac{b_1 b_2}{2}\right)b_0(1 + s^{\alpha})$.  Plugging in the values of $k_1,k_2$, and $k_3$ into Proposition~\ref{prop:mlcostcomp} gives the result.
\end{proof}

\section{A discrete, heuristic MLSVGD algorithm with adaptive stopping criterion}
\label{sec:Algorithm}

\begin{algorithm}[t]
\caption{Discrete MLSVGD with adaptive stopping criterion}
\SetAlgoLined
\textbf{Inputs:} (unnormalized) densities $\pi^{(1)}, \dots, \piL$, initial particles $\{\btheta^{[i]}_0\}_{i = 1}^N$, step size $\delta$, tolerance $\epsilon$\;

\KwResult{Particles $\{\btheta^{[i]}_t\}_{i = 1}^N$ }
 \For{$\ell = 1, \dots, L$}{
  
  \Repeat{$\hat{\bg}^{(\ell)}_t \le \epsilon$}{
  
  Set $s_i = \nabla \log \pil(\btheta^{[i]}_t)$ for $i = 1, \dots, N$\;
  
   \For{$i = 1, \dots, N$}
   {
   
   $\btheta^{[i]}_{t+\delta} = \btheta^{[i]}_t + \frac{\delta}{N} \left(\sum_{j = 1}^N \nabla_1 K(\btheta^{[j]}_t, \btheta_t^{[i]}) + \sum\nolimits_{j = 1}^N K(\btheta^{[j]}_t, \btheta^{[i]}_t) s_j \right)$\;
   
   }
   
   Estimate the norm of the gradient $\hat{\bg}^{(\ell)}_t$ as in \eqref{eq:GradientEstimate}
   
   Set $t \leftarrow t + \delta$
   }
 }
 \label{alg:MLSVGD}
\end{algorithm}

In this section, we propose a discrete, heuristic MLSVGD method given in Algorithm~\ref{alg:MLSVGD} that uses an adaptive stopping criterion to decide when to switch to the next higher level. The proposed Algorithm~\ref{alg:MLSVGD} uses the estimates of the gradient norms to decide when to switch to the next higher level.
Thus, the algorithm avoids requiring any constants that are not readily available in practice. In particular, the algorithm is independent of the constants and rates used in the MLSVGD cost complexity analysis to derive the optimal choice of times $T_1^*, \dots, T_L^*$. 

Let $\bg^{(\ell)}_t$ denote the functional gradient of the KL divergence, as discussed in Section~\ref{sec:SLSVGD:SVGD}, at $\mu^{(\ell)}_t$ with target measure $\pil$.  We approximate the expected norm of the gradient $\bE_{\btheta \sim \mu^{(\ell)}_t}\left\| \bg^{(\ell)}_t(\btheta) \right\|$ with the estimator
\begin{equation}
    \hat{\bg}^{(\ell)}_t = \frac{1}{N} \sum\nolimits_{i=1}^N  \left\| \sum\nolimits_{j = 1}^N \nabla_1 K(\btheta^{[j]}_t, \btheta_t^{[i]}) + \sum\nolimits_{j = 1}^N K(\btheta^{[j]}_t, \btheta^{[i]}_t) \nabla \log \pil (\btheta^{[j]}_t)  \right\| \, ,
\label{eq:GradientEstimate}
\end{equation}
where we note that each term in the sum is computed during the update~\eqref{eq:Prelim:SVGDUpdate}.  The adaptive stopping criteria used in Algorithm~\ref{alg:MLSVGD} is to terminate the iterations at level $\ell$ whenever $\hat{\bg}^{(\ell)}_t \le \epsilon$.  
Ideally, one would want to track the KL divergence between the SVGD approximation and the target distribution and switch to the following level once the KL divergence is below some specified threshold. However, because the normalized target density as well as the density of the SVGD approximation itself are unknown, attempting to monitor the KL divergence at each iteration is impractical. 
The adaptive stopping criteria based on the gradient norm, which we use, is motivated by \cite[Equation 61]{duncan2019geometry}.  It states that for small perturbations from the target density, the KL divergence between the perturbed distribution and the target distribution is asymptotically the same as the norm of the gradient squared.

\section{Numerical experiments}
\label{sec:NumericalResults}

We now demonstrate MLSVGD on Bayesian inverse problems: The aim is to infer the unknown coefficients of a PDE model from noisy observations of the state of the PDE at a few locations in the spatial domain. In Section~\ref{sec:NumRes:DiffReact}, we consider a reaction-diffusion model with unknown reaction parameters, which are then inferred from measurements of the diffusion-reaction field. In Section~\ref{sec:NumRes:EulerBernoulliBeam}, the displacement of an Euler-Bernoulli beam is observed and we then infer the stiffness of the beam. Details about the setup of the numerical experiments are in Appendix~\ref{appdx:NumRes}.

\subsection{Diffusion equation with nonlinear reaction term}
\label{sec:NumRes:DiffReact}

\begin{figure}
\begin{center}
\scalebox{0.94}{
\begin{tabular}{ccc}
\includegraphics[width=0.35\columnwidth]{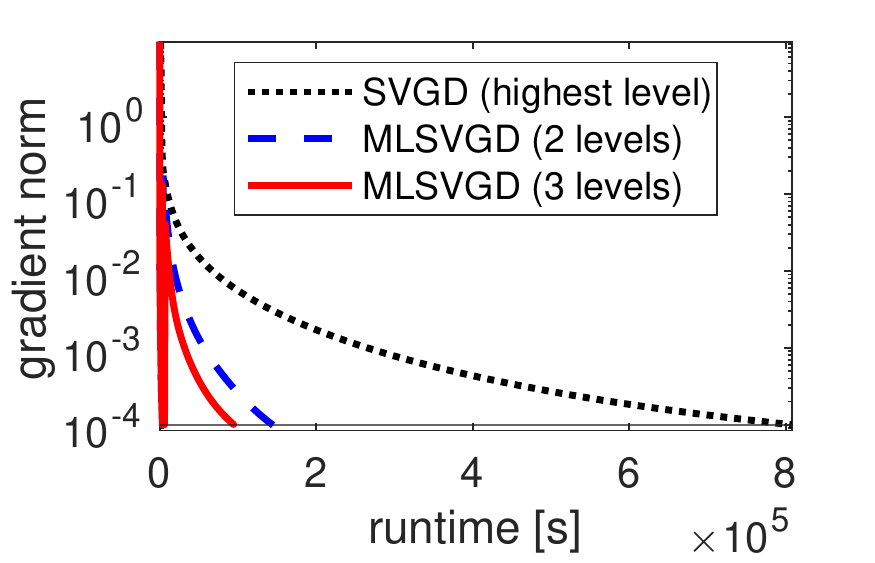} & \hspace*{-0.2cm}\includegraphics[width=0.35\columnwidth]{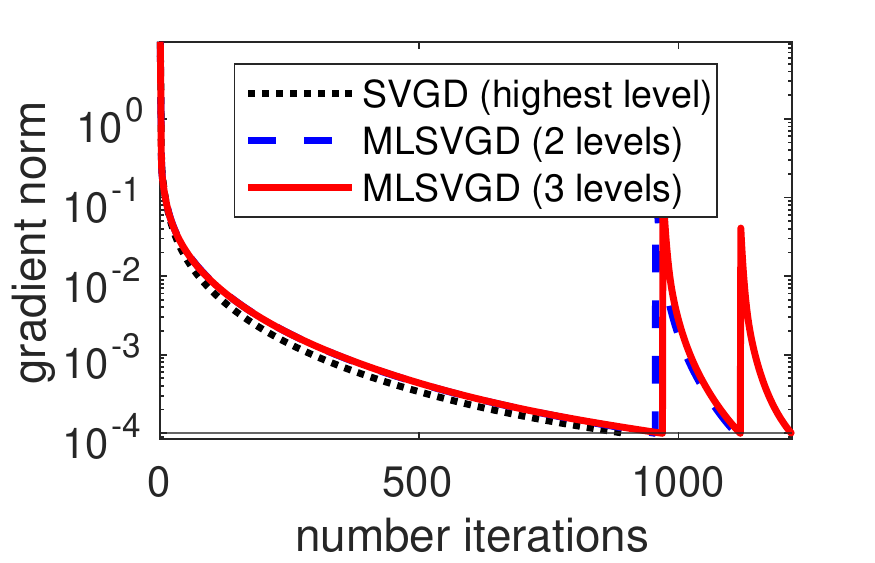} &
\hspace*{-0.2cm}\includegraphics[width=0.30\columnwidth]{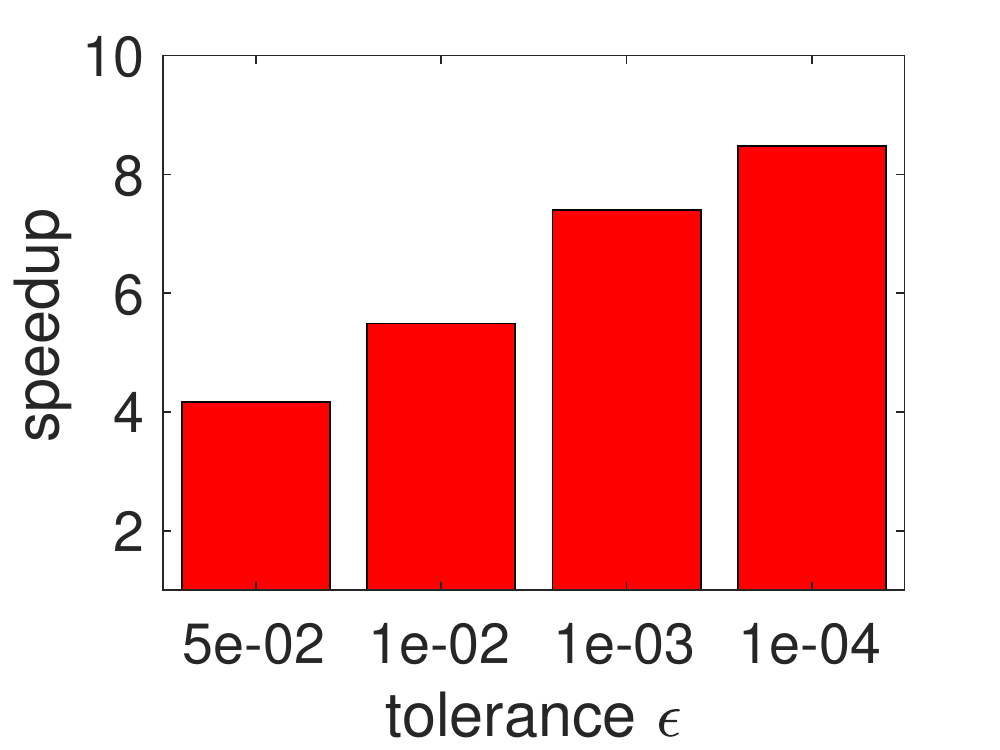}\\
(a) runtime, tolerance $\epsilon = 10^{-4}$ & \hspace*{-0.2cm}(b) iterations, tolerance $\epsilon = 10^{-4}$ & \hspace*{-0.2cm}(c) MLSVGD (3 levels)
\end{tabular}
}
\end{center}
\caption{Diffusion-reaction: MLSVGD achieves speedups because most of the iterations are on lower, cheaper levels, in contrast to SVGD that performs all iterations on the highest, most expensive level. A spike in the gradient norm indicates switching to a higher level.}
\label{fig:NLExp:CostErr}
\end{figure}

Let $\Omega = (0,1)^2$ and $\P = \R^2$ and consider the PDE 
\begin{equation}
-\nabla^2 u(x_1,x_2; \btheta) + g(u(x_1,x_2;\btheta),\btheta) = 100\sin(2\pi x_1)\sin(2\pi x_2)\,,\quad \bx \in \Omega\,,
\label{eq:diffusion_equation}
\end{equation}
with homogeneous Dirichlet boundary conditions, where $\bx = [x_1,x_2]^T$, $\btheta = [\theta_1,\theta_2]^T \in \P $, and $u:\Omega\times\P\to\R$ is the solution function. The nonlinear reaction term $g$ is
$$
g(u(\bx;\btheta),\btheta) = (0.1\sin(\theta_1) + 2)\exp(-2.7\theta_1^2)(\exp(1.8\theta_2u(\bx;\btheta)) - 1)\,.
$$
The PDE \eqref{eq:diffusion_equation} is discretized with finite differences on a grid with equidistant grid points and mesh width $h > 0$. The corresponding system of nonlinear equations is solved with Newton's method and inexact line search based on the Armijo condition. The model $G_{\ell}:\P \to \Y$ derived with mesh width $h  = 2^{-\ell - 2}$ maps from $\P$ into $\Y = \R^{12}$. The components of the observed data $G_{\ell}(\btheta) \in \Y$ correspond to the value of the approximated solution function at the spatial coordinates $[0.25i\,, 0.2j]^T \in\Omega$ with $i \in [3]\,, j\in [4]$. We set $\btheta^* = [-\pi/4,3]^T$ and consider the data $\by = G_{L+1}(\btheta^*) + \boldsymbol e$, where $L = 3$ (i.e., $h = 2^{-5}$) and $\boldsymbol e$ adds zero-mean Gaussian noise of $0.5\%$. The prior distribution is a Gaussian distribution with mean $[\pi/2, 1.5]$ and diagonal covariance matrix with $[50, 0.5]$ on the diagonal.

\begin{figure}
\scalebox{0.95}{
\begin{tabular}{ccc}
\includegraphics[width=0.33\columnwidth]{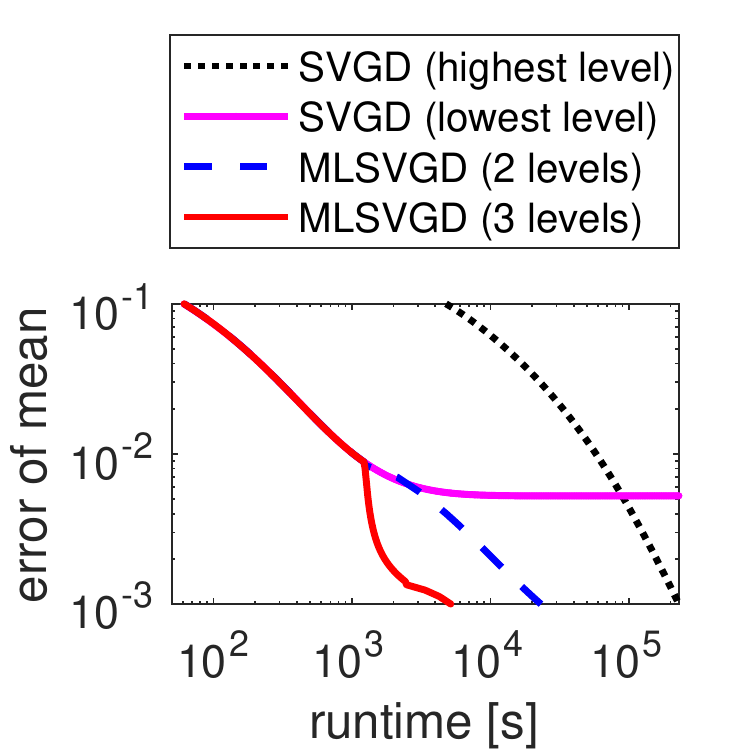} & \hspace*{-0.2cm}\includegraphics[width=0.33\columnwidth]{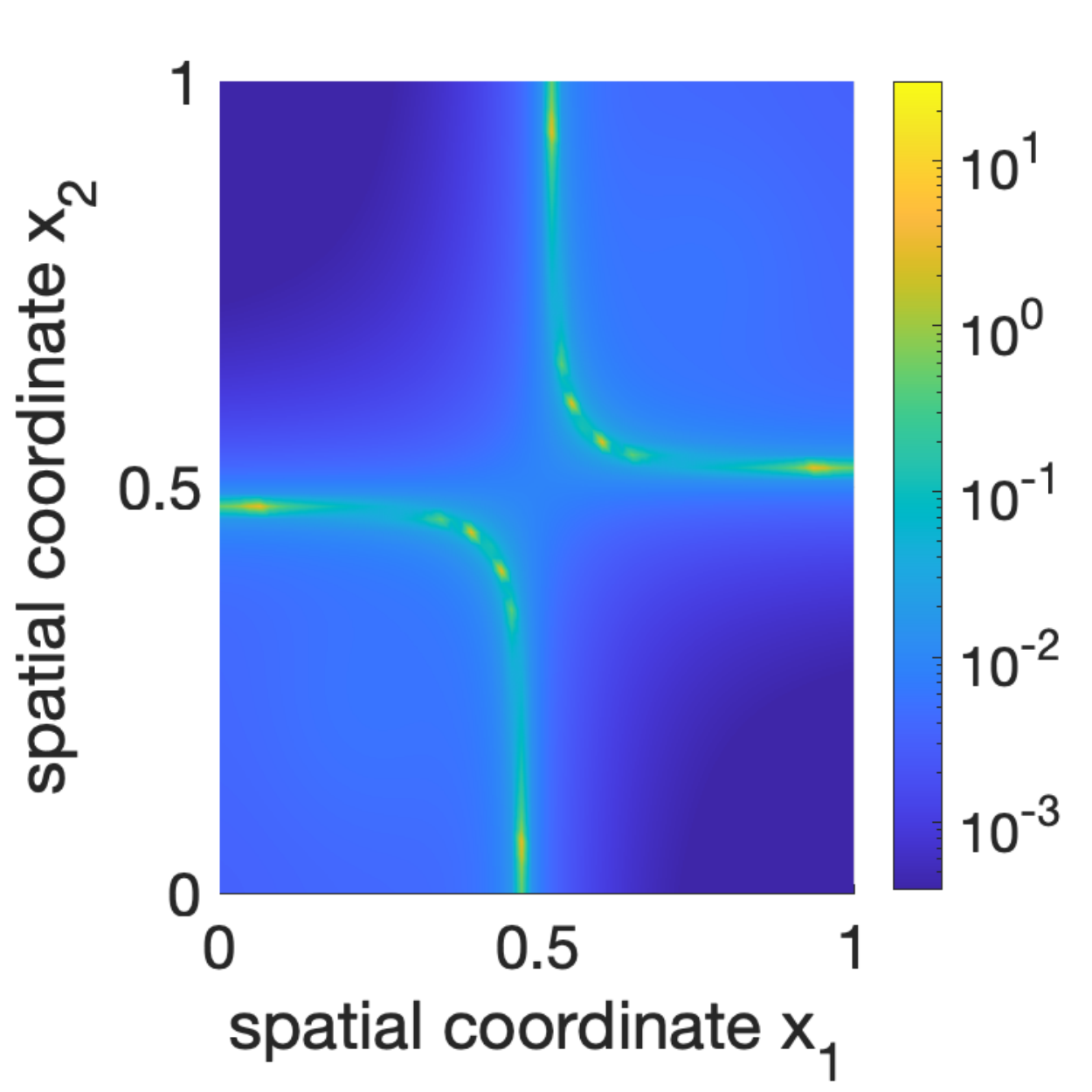} & \hspace*{-0.2cm}\includegraphics[width=0.33\columnwidth]{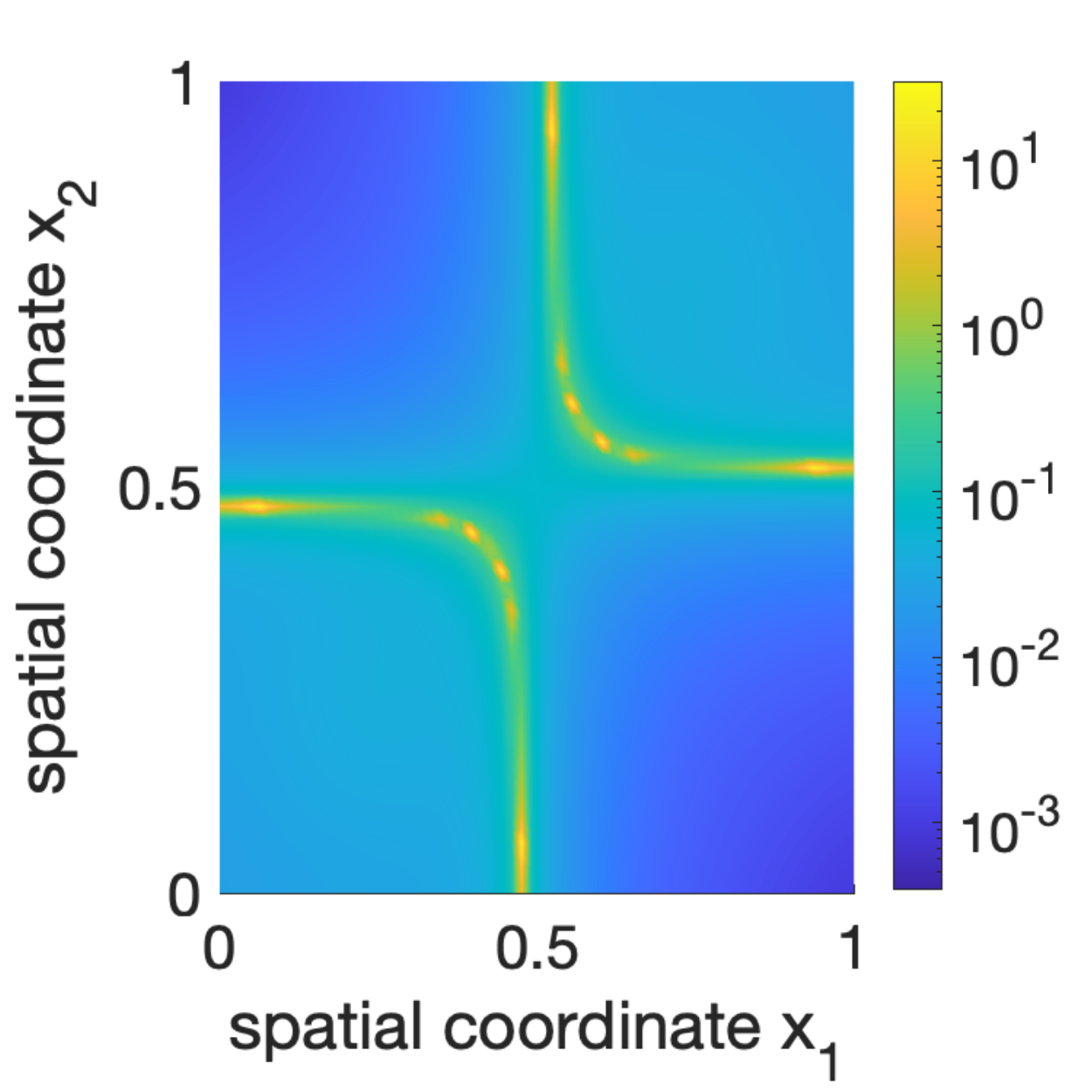}\\
(a) error w.r.t.~MCMC reference & \hspace*{-0.2cm}(b) MLSVGD & \hspace*{-0.2cm}(c) SVGD (same costs as (b))
\end{tabular}
}
\caption{Diffusion-reaction: MLSVGD reaches a particle mean with error $10^{-3}$ with respect to an MCMC reference with more than one order of magnitude speedup compared to SVGD.}
\label{fig:NLExp:MeanError}
\end{figure}

\paragraph{SVGD and MLSVGD} We start with $N = 1000$ particles sampled from a normal distribution with mean $[1, 1]^T$ and diagonal covariance matrix with $10^{-4}$ on the diagonal. The kernel is $k(\btheta,\btheta') = \exp\parr*{- \norm{\btheta - \btheta'}^2 / (2\sigma_k)}$ with $\sigma_k = 10^{-2}$. The gradient of the likelihood is approximated with central differences with mesh width $2^{-6}$. The step size is $\delta = 10^{-1}$. We run SVGD for $\piL$ until the norm of the estimated gradient \eqref{eq:GradientEstimate} reaches a tolerance $\epsilon$. We also run MLSVGD as in Algorithm~\ref{alg:MLSVGD} with levels $\ell \in \{1, 2, 3\}$ and $\ell \in \{1, 3\}$.

\paragraph{Results} Figure \ref{fig:NLExp:CostErr} shows the decay of the estimated gradient norm \eqref{eq:GradientEstimate} for SVGD and MLSVGD with two and three levels, respectively, for a tolerance $\epsilon = 10^{-4}$. While the number of total iterations over all levels in MLSVGD is higher than in SVGD, the costs per iteration are lower on lower levels and thus MLSVGD achieves a speedup of about 8 in this example. Notice that a switch to the next higher level leads to an increase of the gradient norm (e.g., Figure \ref{fig:NLExp:CostErr}b near 1000 iterations), which is then reduced quickly in subsequent iterations. MLSVGD with 2 levels ($\ell \in \{1, 3\}$) achieves a slightly lower speedup than MLSVGD with 3 levels in this example. Figure \ref{fig:NLExp:CostErr}c shows the speedup of MLSVGD with 3 levels for various tolerances. The speedup increases as the tolerance decreases. Figure \ref{fig:NLExp:MeanError}a shows the error of the particle mean with respect to an MCMC reference over 10 replicates (cf.~Appendix~\ref{appdx:NumRes}). The proposed MLSVGD with 3 levels achieves more than one order of magnitude speedup compared to SVGD on the highest level. Notice that running SVGD on the lowest level $\ell = 1$ is fast but leads to a bias of the particle mean as indicated by the leveling off of the corresponding curve. Figure \ref{fig:NLExp:MeanError}b-c show the pointwise error of the finite-difference solution $u$ of \eqref{eq:diffusion_equation} computed at the particle mean of MLSVGD and the particle mean of SVGD with the same costs as MLSVGD. The error is computed with respect to the solution at the MCMC reference. Notice the lighter color in the SVGD plot, which indicates higher pointwise error.

\subsection{Euler-Bernoulli beam}
\label{sec:NumRes:EulerBernoulliBeam}

Let $\Omega = (0,1) \subset \R$ and consider the Euler-Bernoulli beam described by 
\begin{equation}
\partial_x^2(E(x) \partial_x^2 u(x)) = f(x)\,, \quad x \in \Omega\,,
\label{eq:euler_bernoulli_equation}    
\end{equation}

where $u : \Omega \to \R$ is the vertical deflection of the beam and $f : \Omega \to \R$ is the load. The effective stiffness of the beam is given by $E : \Omega\to \R$ and describes the beam geometry and material properties. The beam is in cantilever configuration, where the left boundary is fixed and the right boundary is free.

The observation $\by \in \R^{41}$ is the displacement at $41$ equidistant points in $\Omega$ polluted with $0.01\%$ zero-mean Gaussian noise. We consider a smoothed piecewise constant approximation $\hat{E}_d$ of the stiffness $E$ that depends on $d \in \mathbb{N}$ parameters $\btheta = [\theta_1, \dots, \theta_d]^T$, cf.~Appendix~\ref{appdx:NumRes}. The parameter-to-observable map $G_{\ell}$ is then given by numerically solving \eqref{eq:euler_bernoulli_equation} with  stiffness $\hat{E}_d$ on level $\ell$. The levels $\ell = 1, \dots, 6$ are corresponding to a discretization of the PDE on a mesh of $51, 101, \dots, 501$ equidistant grid points. The prior is log-normal with parameters $\mu = 1$ and $\sigma = 0.05$.  

\begin{figure}
\scalebox{0.97}{
\begin{tabular}{ccc}
\includegraphics[width=0.35\columnwidth]{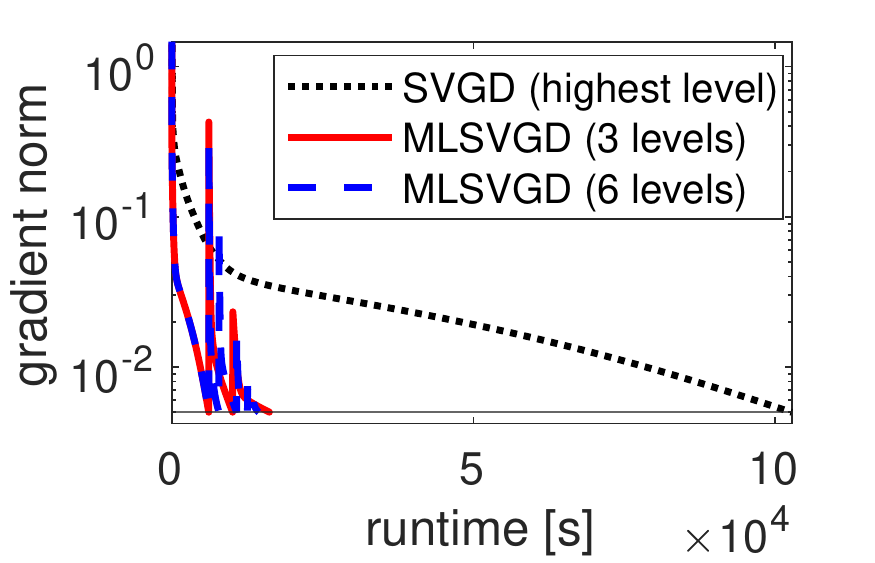} & \hspace*{-0.5cm}\includegraphics[width=0.30\columnwidth]{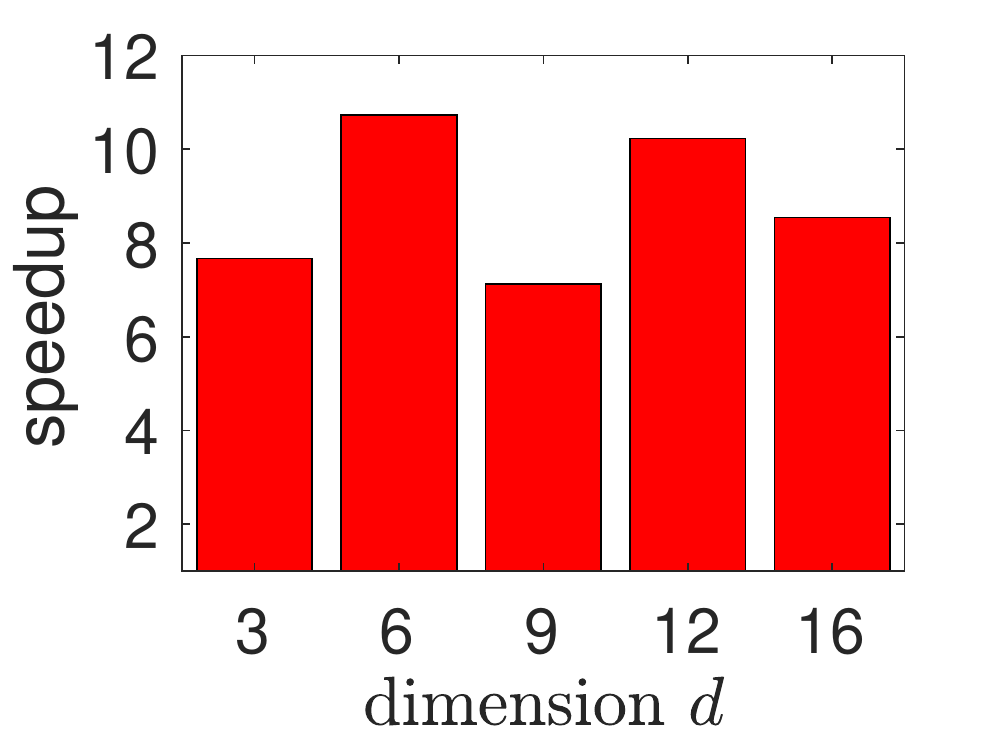}& \hspace*{-0.5cm}\includegraphics[width=0.35\columnwidth]{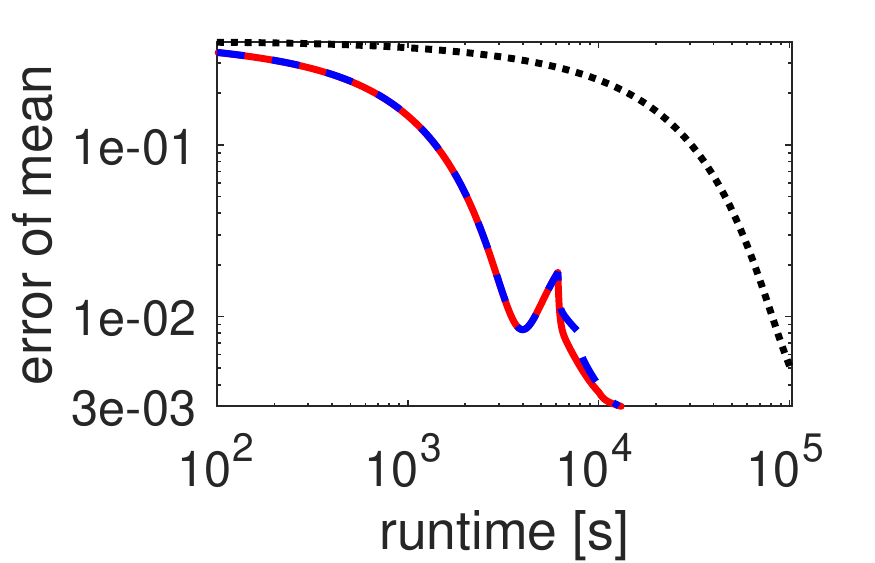} \\
(a) runtime, $d = 9$ &  (b) \parbox{3cm}{speedup (3 levels)} & (c) error w.r.t.~MCMC reference
\end{tabular}
}
\caption{Euler-Bernoulli: MLSVGD achieves speedups between 6--10 in this example compared to SVGD. (Plots (c) shown for $d = 9$.)}
\label{fig:EB:All}
\end{figure}
\paragraph{Results for SVGD and MLSVGD} The initial distribution is normal with mean $[1, 1, \dots, 1]^T \in \mathbb{R}^d$ and diagonal covariance with $4 \times 10^{-4}$ on the diagonal. We consider $N = 500$ particles. The step size is $\delta = 10^{-3}$ for $d = 3$ and $\delta = 10^{-2}$ for $d \in \{6, 9\}$ and $\delta = 5 \times 10^{-3}$ for $d \in \{12, 16\}$. The kernel bandwidth $\sigma_k$ is $10^{-6}$ for $d = 3$ and $10^{-5}$ for $d \in \{6, 9\}$ and $5 \times 10^{-5}$ for $d \in \{12, 16\}$. We consider MLSVGD for levels $\ell \in \{1, \dots, 6\}$ and $\ell \in \{1, 3, 6\}$. The rest of the setup is the same as in Section~\ref{sec:NumRes:DiffReact}. Figure~\ref{fig:EB:All}a shows the convergence behavior of MLSVGD and SVGD for the problem with $d = 9$ dimensions and tolerance $\epsilon = 5 \times 10^{-3}$. A speedup of about 6 is observed to reach an estimated gradient norm below $\epsilon$. Note that MLSVGD with 3 levels achieves about the same speedup as MLSVGD with 6 levels, which indicates that adding more and more intermediate levels cannot further reduce the costs. Speedups are reported in Figure~\ref{fig:EB:All}b for MLSVGD with 3 levels; cf.~Appendix~\ref{appdx:NumRes}. If one asks for the error of the particle mean to be below $3 \times 10^{-3}$ with respect to an MCMC reference, then MLSVGD achieves a speedup of about one order of magnitude compared to SVGD, as shown in Figure~\ref{fig:EB:All}c. 
Figure~\ref{fig:EBVarianceDim9} shows the relative pointwise error of the finite-difference solution $u$ of \eqref{eq:euler_bernoulli_equation} computed at the particles obtained with MLSVGD and single-level SVGD; see also Figure~\ref{fig:EBVariance} in the appendix.  The error bars denote the minimum and maximum pointwise error of the inferred solutions over the ensemble of particles. The results show that MLSVGD achieves a similar error as single-level SVGD even though the computational costs of MLSVGD are lower than single-level SVGD in this example; cf.~Figure~\ref{fig:EB:All}b. Additionally, the variation of the error in terms of minimum and maximum error over the ensemble is comparable between MLSVGD and single-level SVGD in this example.

\begin{figure}
\scalebox{0.95}{
\begin{tabular}{ccc}
 \hspace*{-0.25cm}\includegraphics[width=0.33\columnwidth]{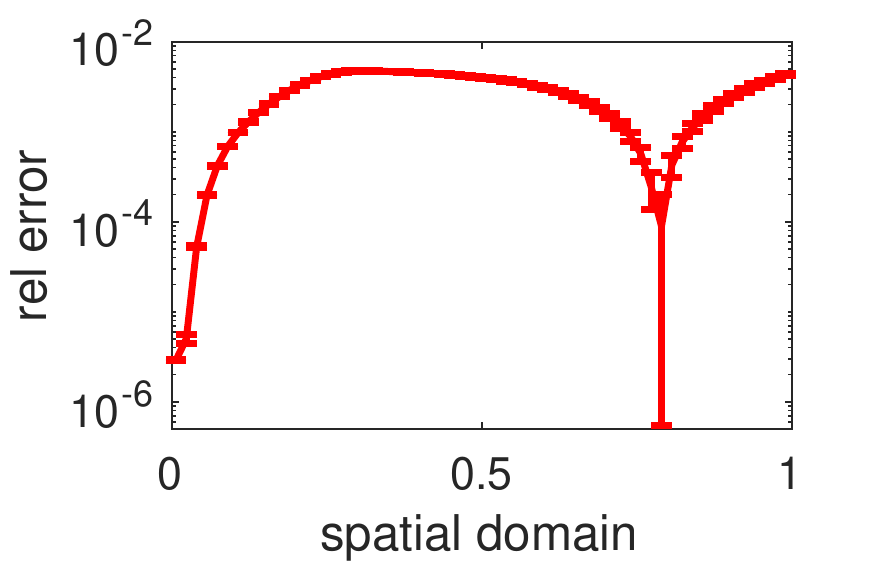} & \hspace*{-0.25cm}\includegraphics[width=0.33\columnwidth]{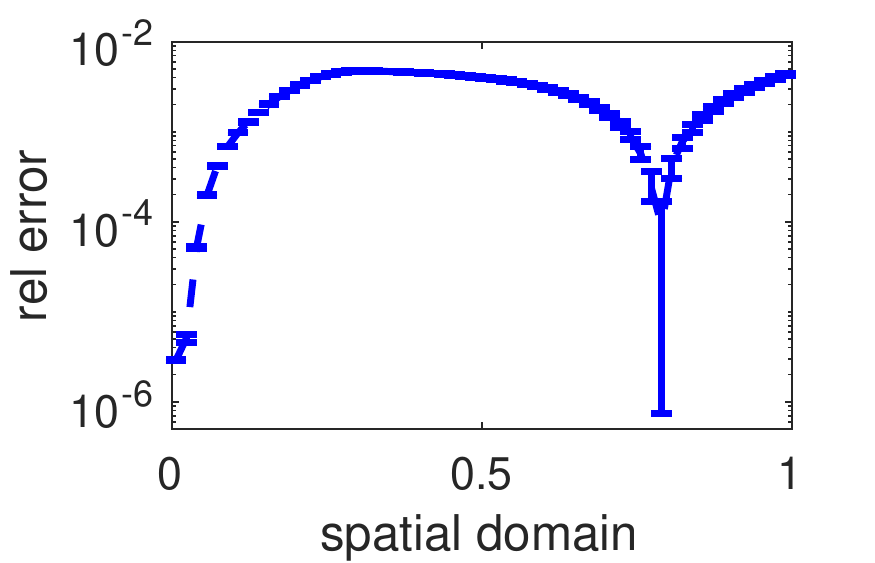} & \includegraphics[width=0.33\columnwidth]{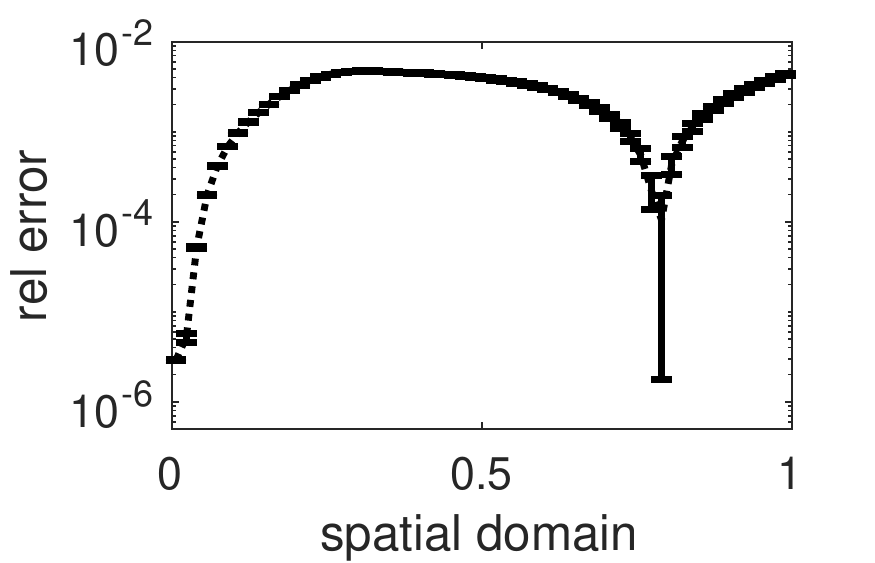}\\
(a) MLSVGD (3 levels)  & \hspace*{-0.25cm}(b) MLSVGD (6 levels)& \hspace*{-0.25cm}(c) single-level SVGD\\
\end{tabular}
}
\caption{Euler-Bernoulli: The pointwise errors over an ensemble of inferred solutions obtained with MLSVGD (left and middle) is comparable to the errors obtained with the computationally more expensive single-level SVGD (right) in this example. The error bars show the minimum and maximum error over the ensemble. Results are shown for $d = 9$.}
\label{fig:EBVarianceDim9}
\end{figure}

\section{Conclusions}
The proposed MLSVGD shows that speedups compared to single-level SVGD can be achieved by balancing the SVGD error with the discretization error given by a hierarchy of ever more accurate and ever more expensive-to-sample distributions. The analysis is conducted in the mean-field limit and shows a cost complexity reduction of MLSVGD compared to single-level SVGD. The numerical experiments demonstrate empirically that MLSVGD achieves up to one order of magnitude speedup compared to single-level SVGD in the discrete-time and finite-particle regime in the applications considered in this work. A cost analysis in discrete time and with finite particles remains future work for MLSVGD especially because there are only limited convergence results available even for single-level SVGD for discrete-time and finite-particle regimes.

\section*{Acknowledgements}
The first and third author were partially funded by the National Science Foundation under grants CMMI-1761068 and IIS-1901091. The first author was additionally supported in part by the Research Training Group in Modeling and
Simulation funded by the National Science Foundation via grant RTG/DMS 1646339. The third author acknowledges additional support from the AFOSR MURI on multi-information sources of multi-physics systems under
Award Number FA9550-15-1-0038 (Dr. Fariba Fahroo).

\bibliography{references}
\bibliographystyle{abbrv}

\appendix 

\section{Metrics and divergences and other definitions}
\label{appx:Metrics}
The Hellinger distance between two probability distributions $\mu$ and $\eta$ on $\R^d$ is defined as
\begin{equation}
    \dhell(\mu, \eta) = \sqrt{\frac{1}{2} \int_{\R^d} \left( \sqrt{\mu(\btheta)} - \sqrt{\eta(\btheta)} \right)^2\ \mathrm{d}\btheta}\, .
    \label{eq:hellinger}
\end{equation}
The Kullback-Leibler (KL) divergence from $\mu$ to $\eta$ is defined as
\begin{equation}
    \KL(\mu || \eta) = \int_{\R^d} \mu(\btheta) \log\left(\frac{\mu(\btheta)}{\eta(\btheta)} \right)\ \mathrm{d}\btheta \, .
\label{eq:kldiv}
\end{equation}
Define the $L^2(\mu)$ space for a distribution $\mu$ and vector-valued functions as
\begin{equation}
	L^2(\mu) = \left\{ f : \int_{\R^d} \|f(\btheta)\|^2 \mu(\btheta) \ \mathrm{d}\btheta < \infty \right\}\, ,
\end{equation}
and the $L^2(\mu)$ norm of a vector-valued function as
\begin{equation}
	\|f\|_{L^2(\mu)}^2 = \int_{\R^d} \|f(\btheta)\|^2 \mu(\btheta) \ \mathrm{d}\btheta\,.
	\label{eq:appx:L2Norm}
\end{equation}

\section{Lemma~\ref{lmma:klmodelbound} and proof}
\label{appdx:proof1}

\begin{lemma}
If Assumption~\ref{asm:modelerror} holds, there exists a constant $C > 0$ such that for all $1 \le \ell_1, \ell_2 \le \infty$ sufficiently large
\begin{equation}
	\text{\emph{KL}}(\pi^{(\ell_1)}\ ||\ \pi^{(\ell_2)}) \le C \|G_{\ell_1} - G_{\ell_2}\|_{L^2(\pi_0)}\, .
\end{equation}
Note that for $\ell = \infty$ we say $G_{\ell} = G$.
\label{lmma:klmodelbound}
\end{lemma}
We note that this proof closely mirrors the proofs of Lemmas 4.2 and 4.3 in~\citep{MX}, but is slightly more general.

\begin{proof}
For brevity write $G_i = G_{\ell_i}$, $Z_i = Z_{\ell_i}$, and $\pi_i = \pi^{(\ell_i)}$ for $i=1,2$.  Consider that for any vectors $\bu,\bv,\bw \in \R^d$ and symmetric positive definite matrix $\bA \in \R^{d\times d}$ we have
\begin{equation}
\begin{split}
	\| \bu - \bw \|_{\bA}^2 - \| \bv - \bw \|_{\bA}^2  &= \| (\bu - \bv) + (\bv - \bw)\|_{\bA}^2 - \|\bv - \bw \|_{\bA}^2 \\
	&= \langle (\bu - \bv),\ \bA(\bu - \bv) \rangle + 2\langle (\bu - \bv),\ \bA(\bv - \bw) \rangle \\
	&= \langle (\bu - \bv),\ \bA( \bu + \bv - 2\bw ) \rangle \\
	&\le \|\bu - \bv\| \cdot \| \bA (\bu + \bv - 2\bw) \|\, ,
\end{split}
\label{eq:bound1}
\end{equation}
with the last line following from the Cauchy-Schwarz inequality.  Applying this bound with $\bu = G_1(\btheta)$, $\bv = G_2(\btheta)$, $\bw = \by$, and $\bA = \bGamma^{-1}$ gives
\begin{equation}
\begin{split}
& \int_{\Theta} \left| \| \by - G_1(\btheta)\|_{\bGamma^{-1}}^2 - \| \by - G_2(\btheta)\|_{\bGamma^{-1}}^2  \right| \pi_0(\btheta) \ \mathrm{d}\btheta \\
&\le \int_{\Theta} \| G_1(\btheta) - G_2(\btheta) \| \cdot \| \bGamma^{-1}(2\by - G_1(\btheta) - G_2(\btheta) ) \| \pi_0(\btheta)\ \mathrm{d}\btheta \\
&\le \| G_1 - G_2\|_{L^2(\pi_0)} \cdot \|\bGamma^{-1}(2\by - G_1 - G_2 )\|_{L^2(\pi_0)} \, ,
\end{split}
\label{eq:bound2}
\end{equation}
where the last line again follows from the Cauchy-Schwarz inequality on the inner-product space $L^2(\pi_0)$.  The KL divergence can now be bounded using Equation~\eqref{eq:bound2}
\begin{equation}
\begin{split}
	\text{KL}(\pi_1\ ||\ \pi_2) &= \int_{\Theta} \pi_1(\btheta) \log\left( \frac{\pi_1(\btheta)}{\pi_2(\btheta)} \right)\ \mathrm{d}\btheta \\
	&= \int_{\Theta} \pi_1(\btheta) \log \left(  \frac{Z_2 \exp\left( -\frac{1}{2}\|\by - G_1(\btheta)\|_{\bGamma^{-1}}^2 \right)}{Z_1 \exp\left( -\frac{1}{2}\|\by - G_2(\btheta)\|_{\bGamma^{-1}}^2 \right) }  \right) \ \mathrm{d}\btheta \\
	&= \log\left(\frac{Z_2}{Z_1}\right) +  \int_{\Theta} \pi_1(\btheta) \log \left(  \frac{\exp\left( -\frac{1}{2}\|\by - G_1(\btheta)\|_{\bGamma^{-1}}^2 \right)}{ \exp\left( -\frac{1}{2}\|\by - G_2(\btheta)\|_{\bGamma^{-1}}^2 \right) }  \right) \ \mathrm{d}\btheta \\
	&\le \log \left( \frac{Z_2}{Z_1} \right) + \frac{1}{2 Z_1} \int_{\Theta} \left| \| \by - G_1(\btheta)\|_{\bGamma^{-1}}^2 - \| \by - G_2(\btheta)\|_{\bGamma^{-1}}^2  \right| \pi_0(\btheta) \ \mathrm{d}\btheta \\
	&\le \left| \log\left( \frac{Z_2}{Z_1} \right) \right| + \frac{1}{2 Z_1} \| G_1 - G_2\|_{L^2(\pi_0)} \cdot \|\bGamma^{-1}(2\by - G_1 - G_2 )\|_{L^2(\pi_0)} \, ,
\end{split}
\label{eq:bound3}
\end{equation}
where in the second-to-last line we used the fact that $\frac{1}{2}\|\by - G_1(\btheta)\|_{\bGamma^{-1}}^2 \ge 0$ and hence 
\begin{equation}
    \exp\left( -\frac{1}{2}\|\by - G_1(\btheta)\|_{\bGamma^{-1}}^2 \right) \le 1 \, .
\end{equation}
We bound the logarithm of the ratio of the normalizing constants by first bounding the difference of the normalizing constants using the bound in Equation~\eqref{eq:bound2}
\begin{equation}
\begin{split}
	|Z_1 - Z_2| &= \left| \int_{\Theta} \left\{ \exp\left( -\frac{1}{2}\| \by - G_1(\btheta)\|_{\bGamma^{-1}}^2 \right) - \exp\left( -\frac{1}{2}\| \by - G_2(\btheta)\|_{\bGamma^{-1}}^2 \right) \right\} \pi_0(\btheta) \ \mathrm{d}\btheta \right|\\
	&\le \int_{\Theta} \left| \exp\left( -\frac{1}{2}\| \by - G_1(\btheta)\|_{\bGamma^{-1}}^2 \right) - \exp\left( -\frac{1}{2}\| \by - G_2(\btheta)\|_{\bGamma^{-1}}^2 \right) \right| \pi_0(\btheta) \ \mathrm{d}\btheta \\
	&\le \frac{1}{2} \int_{\Theta} \left| \| \by - G_1(\btheta)\|_{\bGamma^{-1}}^2 - \| \by - G_2(\btheta)\|_{\bGamma^{-1}}^2  \right| \pi_0(\btheta) \ \mathrm{d}\btheta \\
	&\le \frac{1}{2} \| G_1 - G_2\|_{L^2(\pi_0)} \cdot \|\bGamma^{-1}(2\by - G_1 - G_2 )\|_{L^2(\pi_0)}\, .
\end{split}
\label{eq:bound4}
\end{equation}
The third line follows from the fact that $|e^{-x} - e^{-y}| \le |x-y|$ for all $x,y\ge0$.  Let $\gamma_{\min} > 0$ denote the smallest eigenvalue of the noise covariance matrix $\bGamma$.  By the triangle inequality
\begin{equation}
\begin{split}
\| \bGamma^{-1}(2\by - G_1 - G_2) \|_{L^2(\pi_0)} 
&\le 2 \|\bGamma^{-1} \by \|_{L^2(\pi_0)} + \| \bGamma^{-1}( G_1 + G_2) \|_{L^2(\pi_0)}  \\
&\le 2 \|\bGamma^{-1} \by \|_{L^2(\pi_0)} + 2\|\bGamma^{-1} G\|_{L^2(\pi_0)} + \| \bGamma^{-1}( G_1 + G_2 - 2G) \|_{L^2(\pi_0)}  \\
&\le 2 \|\bGamma^{-1} \by \|_{L^2(\pi_0)} + 2\|\bGamma^{-1} G\|_{L^2(\pi_0)} \\
&\quad + \frac{1}{\gamma_{\min}}\| G_1 - G \|_{L^2(\pi_0)} +  \frac{1}{\gamma_{\min}}\| G_2 - G \|_{L^2(\pi_0)} \, .
\end{split}
\label{eq:bound5}
\end{equation}
Since $\| G_{\ell} - G \|_{L^2(\pi_0)} \to 0$ by Assumption~\ref{asm:modelerror}, we can bound $\|G_1 - G\|_{L^2(\pi_0)}$ and $\|G_2 - G\|_{L^2(\pi_0)}$ independently of $\ell_1$ and $\ell_2$.  Therefore, there exists a constant $b_1 > 0$ independent of $\ell$ such that
\begin{equation}
\| \bGamma^{-1}(2\by - G_1 - G_2) \|_{L^2(\pi_0)}  \le b_1 .
\label{eq:bound6}
\end{equation}
Combining Equations~\eqref{eq:bound4} and~\eqref{eq:bound6} yields
\begin{equation}
	|Z_1 - Z_2| \le \frac{b_1}{2} \|G_1 - G_2\|_{L^2(\pi_0)} \, .
\label{eq:bound7}
\end{equation}
The ratio of the normalizing constants can be written
\begin{equation}
\left| \frac{Z_2}{Z_1} - 1 \right| = \frac{1}{Z_1} \left| Z_1 - Z_2 \right| \, ,
\label{eq:bound8}
\end{equation}
so the logarithm can be bounded as
\begin{equation}
\left| \log \left( \frac{Z_2}{Z_1} \right) \right| \le \max \left\{ \left| \log\left( 1 - \frac{|Z_2 - Z_1|}{Z_1} \right)  \right|,\  \log\left( 1 + \frac{|Z_2 - Z_1|}{Z_1} \right)   \right\} 
\label{eq:bound9}
\end{equation}
since $x \mapsto |\log x|$ is decreasing on $(0, 1]$ and increasing on $[1,\infty)$.    Combining this with the inequality that $\frac{x}{1+x} \le \log(1 + x) \le x$ for all $x > -1$ gives
\begin{equation}
\begin{split}
\left| \log \left( \frac{Z_2}{Z_1} \right) \right| &\le \max \left\{ \frac{\frac{|Z_2 - Z_1|}{Z_1}}{ 1 - \frac{|Z_2 - Z_1|}{Z_1} } ,\ \frac{|Z_2 - Z_1|}{Z_1} \right\} \le \frac{|Z_1 - Z_2|}{Z_1 - |Z_1 - Z_2|} \, .
\end{split}
\label{eq:bound10}
\end{equation}
Since $Z_{\ell} \to Z \in (0, \infty)$ is a convergent sequence, there exists a constant $b_2 > 0$ such that
\begin{equation}
	Z_1^{-1} \le \sup_{\ell \ge 1} Z_{\ell}^{-1} \le b_2\, .
\label{eq:bound11}
\end{equation} 
Moreover, for all $\ell_1,\ell_2$ sufficiently large $|Z_1 - Z_2| \le b_2^{-1}/2$.  Using the bound gives
\begin{equation}
\left| \log \left( \frac{Z_2}{Z_1} \right) \right| \le \frac{|Z_1 - Z_2|}{b_2^{-1} - |Z_1 - Z_2|}  \le 2 b_2 | Z_1 - Z_2|\, .
\label{eq:bound12}
\end{equation}
Combining Equations~\eqref{eq:bound3},~\eqref{eq:bound6},~\eqref{eq:bound7},~\eqref{eq:bound11}, and~\eqref{eq:bound12} gives
\begin{equation}
	\text{KL}(\pi_1\ ||\ \pi_2) \le \frac{3}{2}b_1 b_2 \|G_1 - G_2\|_{L^2(\pi_0)} \, .
\label{eq:bound13}
\end{equation}
Now set $C = \frac{3}{2}b_1 b_2$ to obtain the result.
\end{proof}

\section{Lemma~\ref{lmma:bipklrate} and proof}
\label{appdx:proof2}

\begin{lemma}
If Assumption~\ref{asm:modelerror} holds with $\psi(\ell) = b_0 s^{-\alpha \ell}$, then Assumptions~\ref{asm:BiasRate},~\ref{asm:MLBiasRate} also hold with the same rate $\alpha$.
\label{lmma:bipklrate}
\end{lemma}

\begin{proof}
Let $\ell_1 = \ell$ and $\ell_2 = \infty$, so that by Lemma~\ref{lmma:klmodelbound} in Appendix~\ref{appdx:proof1} we immediately have
\begin{equation}
	\text{KL}(\pi^{(\ell)}\ ||\ \pi) \le C \|G_{\ell} - G\|_{L^2(\pi_0)} \le C \psi(\ell)  = C b_0 s^{-\alpha \ell} ,
\label{eq:bound14}
\end{equation}
so that $k_1 = Cb_0$.  Moreover, setting $\ell_1 = \ell - 1$ and $\ell_2 = \ell$ and using the triangle inequality gives
\begin{equation}
	\text{KL}(\pi^{(\ell - 1)}\ ||\ \pi^{(\ell)}) \le C \|G_{\ell - 1} - G_{\ell} \|_{L^2(\pi_0)} \le C \left( \|G_{\ell - 1} - G\|_{L^2(\pi_0)} + \|G_{\ell } - G\|_{L^2(\pi_0)} \right)\, .
\end{equation}
Thus,
\begin{equation}
	\text{KL}(\pi^{(\ell - 1)}\ ||\ \pi^{(\ell)}) \le C \left( 1 + \frac{\psi(\ell-1)}{\psi(\ell)} \right) \psi(\ell) \le C b_0 \left( 1 + s^{\alpha} \right) s^{-\alpha \ell}\, , 
\end{equation}
so that $k_2 = C b_0 \left( 1 + s^{\alpha} \right)$.
\end{proof}

\section{A triangle-like inequality for the KL divergence}
\label{appdx:kltriangle}

Let $\rho_0, \rho_1, \rho_2$ be three probability distributions on $\Theta$.  We have that
\begin{equation}
\begin{split}
      \KL(\rho_0 \ ||\ \rho_2) &= \int_{\Theta} \rho_0(\btheta)\log\left( \frac{\rho_0(\btheta)}{\rho_2(\btheta)} \right)\ \mathrm{d}\btheta   \\
      &= \int_{\Theta} \rho_0(\btheta) \log\left( \frac{\rho_0(\btheta) \rho_1(\btheta)}{\rho_1(\btheta)\rho_2(\btheta)} \right) \ \mathrm{d}\btheta \\
      &= \int_{\Theta} \rho_0(\btheta)\log\left( \frac{\rho_0(\btheta)}{\rho_1(\btheta)} \right)\ \mathrm{d}\btheta + \int_{\Theta} \rho_0(\btheta)\log\left( \frac{\rho_1(\btheta)}{\rho_2(\btheta)} \right)\ \mathrm{d}\btheta \\
      &= \KL(\rho_0\ ||\ \rho_1) + \KL(\rho_1\ ||\ \rho_2) + \int_{\Theta} (\rho_0(\btheta) - \rho_1(\btheta)) \log \left( \frac{\rho_1(\btheta)}{\rho_2(\btheta)} \right)\ \mathrm{d}\btheta \, ,
\end{split}
\end{equation}
cf.~the inequality given in \citep{MX}. We refer to this third term in the last line as the remainder term.

\section{Details about numerical experiments}
\label{appdx:NumRes}

\subsection{General}
The step size $\delta$ and kernel bandwidth $h$ was chosen via a manual process so that SVGD on the highest level numerically converged. The same $\delta$ and $h$ are used for SVGD and MLSVGD. Time measurements were performed on compute nodes with Intel Xeon CPU E5-2690 v2, restricted to 8 cores and 32GB memory, with a Matlab implementation. The MCMC reference is computed with the delayed-rejection adaptive Metropolis (DRAM) method \citep{haario2001,10.1007/s11222-006-9438-0} on the highest level $L$ of the respective problem. The covariance matrix of the Gaussian proposal is initialized to be diagonal with $10^{-2}$ on the diagonal. The burn-in time is 10,000 samples. Another 20,000 samples are generated and every other sample is then used to compute the MCMC reference mean $\bar{\btheta}$ of the parameter. The error reported in Figure~\ref{fig:NLExp:MeanError}a and Figure~\ref{fig:EB:All}c is $\frac{1}{10}\sum_{i = 1}^{10} \|\bar{\btheta} - \btheta^{(i)}\|_2$, where $\btheta^{(i)}$ is the mean of (ML)SVGD particles of the $i$-th replicate.

\subsection{Diffusion equation with nonlinear reaction term}
We repeat the experiments of Section~\ref{sec:NumRes:DiffReact} with $N \in \{500, 2500, 5000\}$ particles and show the corresponding speedups in Figure~\ref{fig:DiffReactDifferentN}. The speedup of MLSVGD is roughly the same over the different numbers of particles, which is expected because the cost of MLSVGD scales with the number of particles as the cost of SVGD.

\begin{figure}
\scalebox{0.98}{
\begin{tabular}{ccc}
\includegraphics[width=0.33\columnwidth]{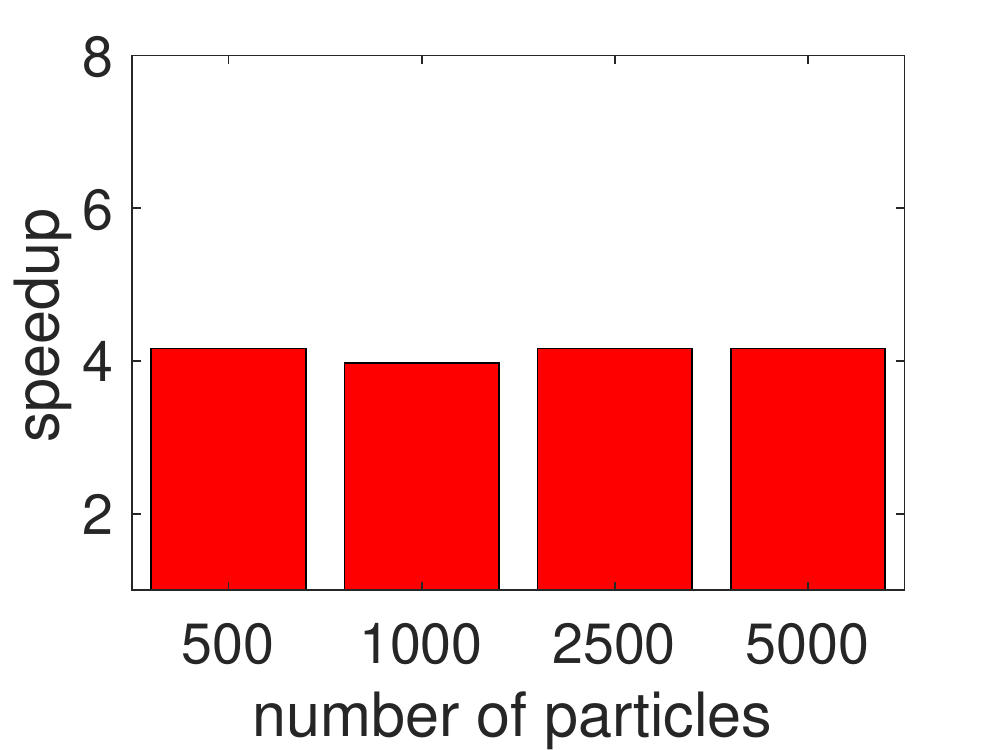}
& \hspace*{-0.5cm}\includegraphics[width=0.33\columnwidth]{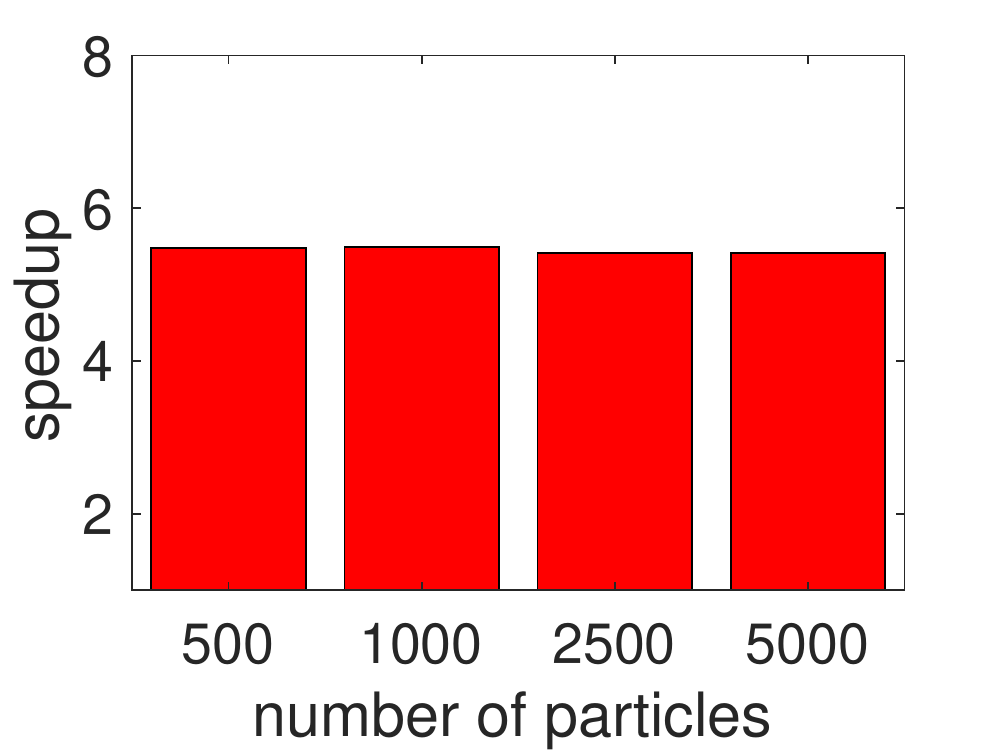}
& \hspace*{-0.5cm}\includegraphics[width=0.33\columnwidth]{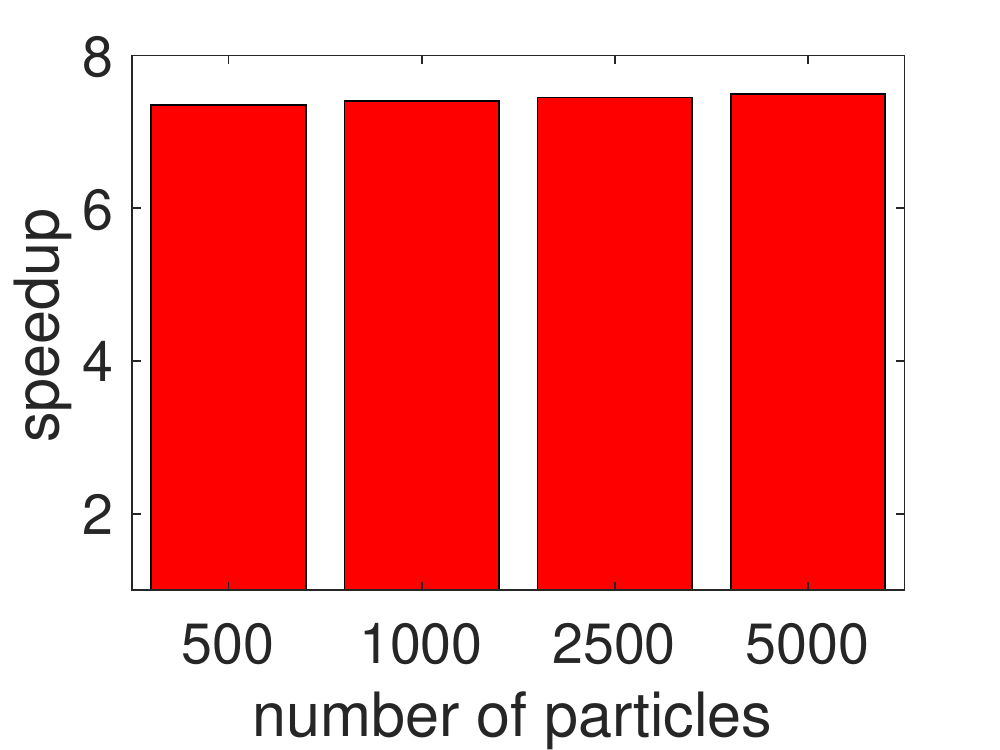}\\
(a) tolerance $\epsilon = 5 \times 10^{-2}$ & (b) tolerance $\epsilon = 10^{-2}$ & (c) tolerance $\epsilon = 10^{-3}$
\end{tabular}
}
\caption{Diffusion-reaction: The cost of MLSVGD scales with the number of particles as the cost of SVGD, which means that the speedups that MLSVGD obtains compared to SVGD in this example remain roughly the same for different number of particles.}
\label{fig:DiffReactDifferentN}
\end{figure}

\subsection{Euler-Bernoulli beam}

In Section \ref{sec:NumRes:EulerBernoulliBeam} we consider the PDE \eqref{eq:euler_bernoulli_equation}
for $\Omega = [0,1]$, where $u : \Omega \to \R$ is the vertical deflection of the beam and $f : \Omega \to \R$ is the load. The effective stiffness of the beam is given by $E : \Omega\to \R$ and describes beam geometry and material properties. The beam is in cantilever configuration, where the left boundary is fixed and the right boundary is free i.e., the boundary conditions are
$$
u(0)= 0 \,, \quad \frac{\partial}{\partial x}u\Big|_{x=0} = 0\,, \quad \frac{\partial^3}{\partial x^3}u\Big|_{x=1} = 0\,, \quad \frac{\partial^3}{\partial x^3}u\Big|_{x=1} = 0\,.
$$
We use the same stiffness $E$ available in the model developed by Matthew Parno for the 2018 Gene Golub SIAM Summer School on ``Inverse Problems: Systematic Integration of Data with Models under Uncertainty.'' The model is available on GitHub.\footnote{ https://github.com/g2s3-2018/labs} 

\paragraph{Forward model} The forward model is derived as follows. Consider the function $I : \R \times \Omega \to \R$ defined as
$$
I(x,\alpha) = \parr*{ 1 + \mathrm{exp}\parr*{-\frac{x-\alpha}{0.005}}}^{-1}\,,
$$
with 
$$
\lim_{x\to-\infty} I(x,\alpha) = 0\,,\quad \lim_{x\to\infty} I(x,\alpha) = 1
$$
such that there is a smooth transition from $0$ to $1$ at $\alpha$. For $k > 1$, let $\alpha_1,\dots,\alpha_{k+1}$ be $k+1$ equidistant points in $\Omega$. Let $\R_+ = \{ z \in \R \st z > 0\}$ and consider the parameter $\btheta = [\theta_1,\dots,\theta_k]^T \in \R^k_+$. Define the function $\hat{E}_i:\Omega\times\R\to\R$ as 
$$
\hat{E}_i(x,\theta_i) = \parr*{1 - I(x,\alpha_i)}\hat{E}_i(x,\theta_{i-1})
$$
Given a parameter $\btheta$, the function $\hat{E}_k$ is a smooth approximation of
the piecewise constant function\\ $\sum_{i=1}^k \theta_i \mathbbm{1}(\alpha_i,\alpha_{i+1}]$, where $\mathbbm{1}(\alpha_i,\alpha_{i+1}]$ is the indicator function of the interval $(\alpha_i,\alpha_{i+1}] \subset \R$. 

\paragraph{Additional plots for $d \in \{3, 6, 12, 16\}$} Figure~\ref{fig:EBD3}--\ref{fig:EBD16} show the analogous results to Figure~\ref{fig:EB:All} for dimension $d \in \{3, 6, 12, 16\}$, respectively. Figure~\ref{fig:EBVariance} shows the analogous results to Figure~\ref{fig:EBVarianceDim9}. The behavior of MLSVGD compared to SVGD is qualitatively the same as for dimension $d = 9$.

\begin{figure}
\scalebox{0.94}{
\begin{tabular}{ccc}
\includegraphics[width=0.33\columnwidth]{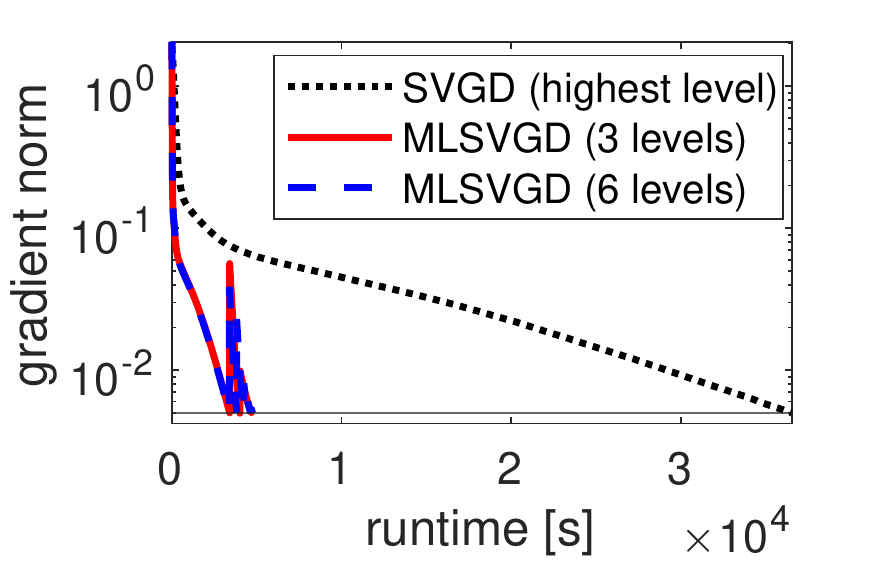} & \hspace*{-0.2cm}\includegraphics[width=0.33\columnwidth]{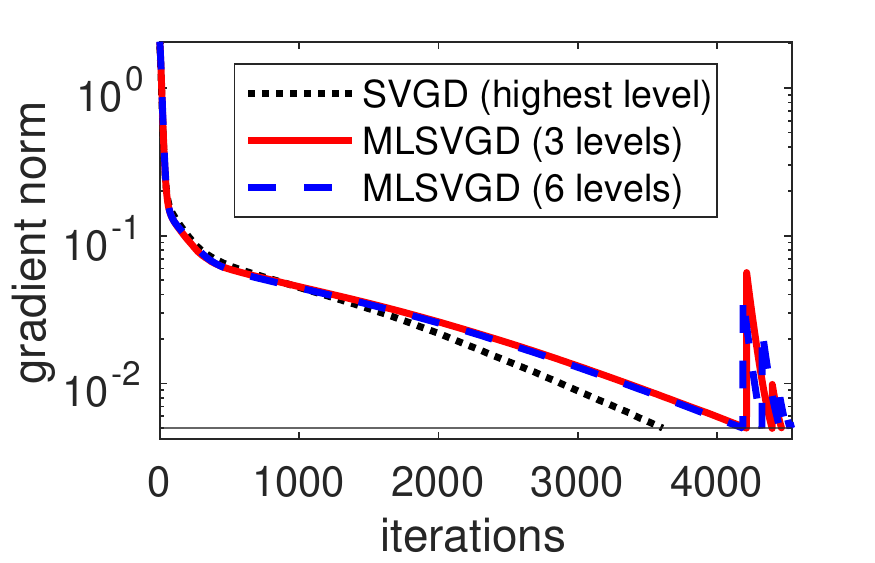} & \hspace*{-0.2cm}\includegraphics[width=0.33\columnwidth]{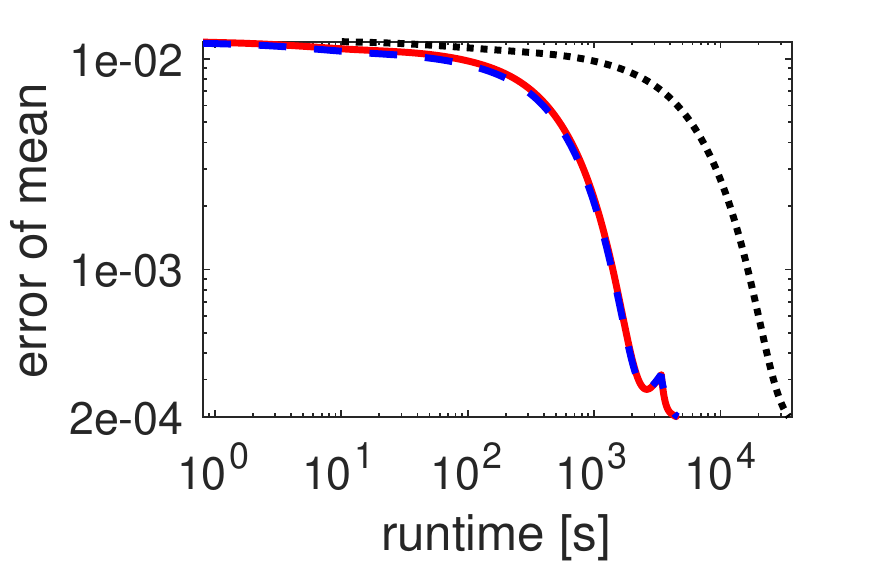}\\
(a) runtime & \hspace*{-0.2cm}(b) iterations & \hspace*{-0.2cm}(c) comparison to MCMC reference
\end{tabular}
}
\caption{Euler-Bernoulli beam: Results of MLSVGD for dimension $d = 3$.}
\label{fig:EBD3}
\end{figure}

\begin{figure}
\scalebox{0.94}{
\begin{tabular}{ccc}
\includegraphics[width=0.33\columnwidth]{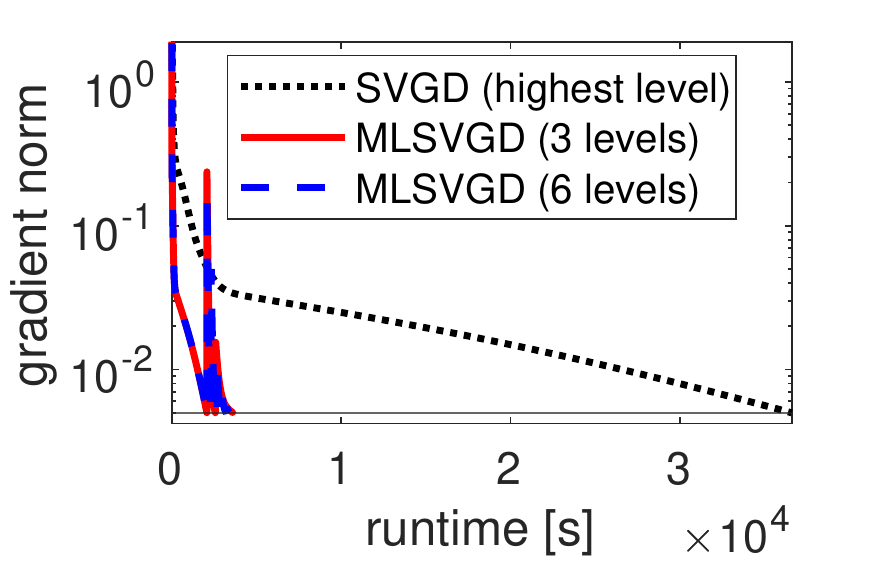} & \hspace*{-0.2cm}\includegraphics[width=0.33\columnwidth]{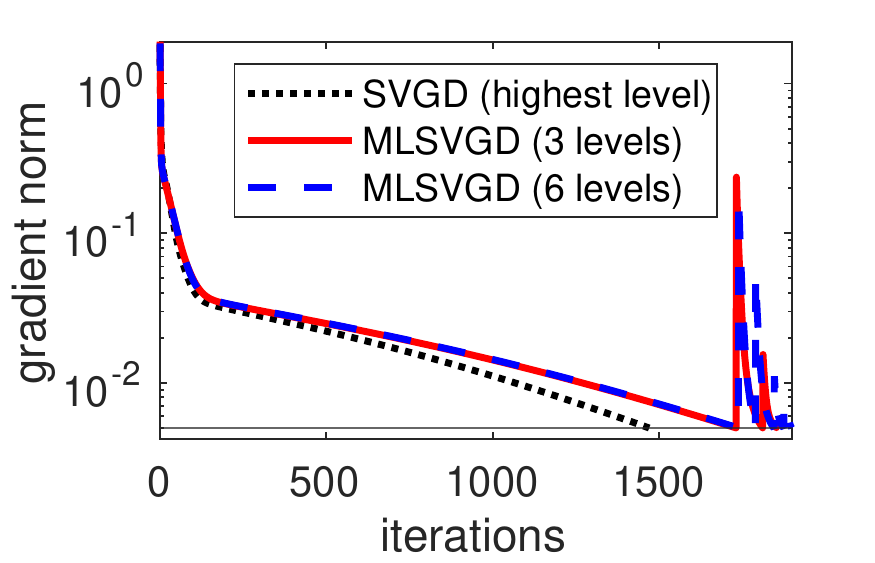} & \hspace*{-0.2cm}\includegraphics[width=0.33\columnwidth]{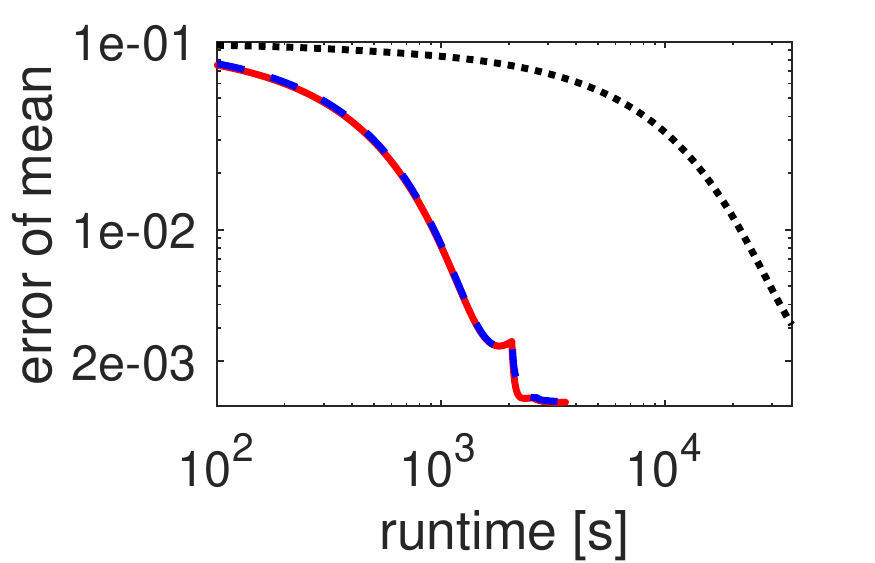}\\
(a) runtime & \hspace*{-0.2cm}(b) iterations & \hspace*{-0.2cm}(c) comparison to MCMC reference
\end{tabular}
}
\caption{Euler-Bernoulli beam: Results of MLSVGD for dimension $d = 6$.}
\label{fig:EBD6}
\end{figure}

\begin{figure}
\scalebox{0.94}{
\begin{tabular}{ccc}
\includegraphics[width=0.33\columnwidth]{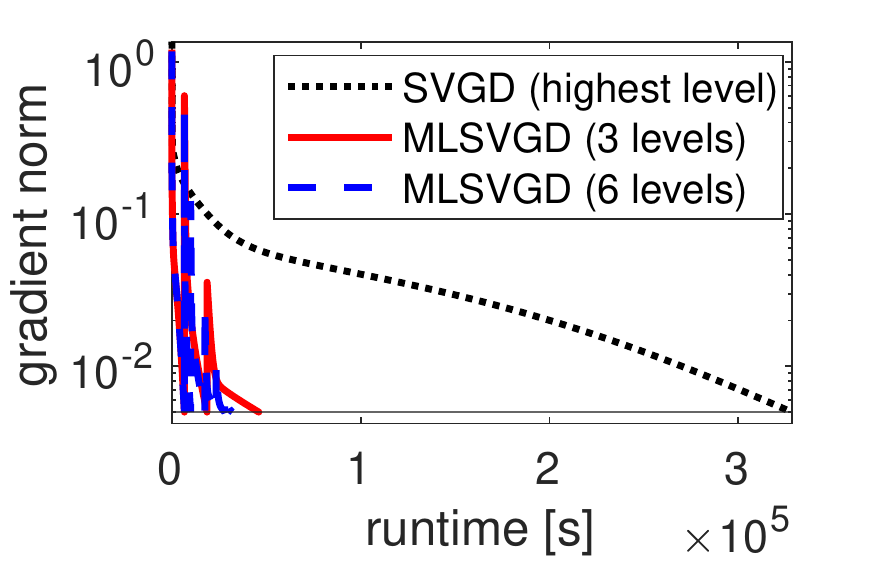} & \hspace*{-0.2cm}\includegraphics[width=0.33\columnwidth]{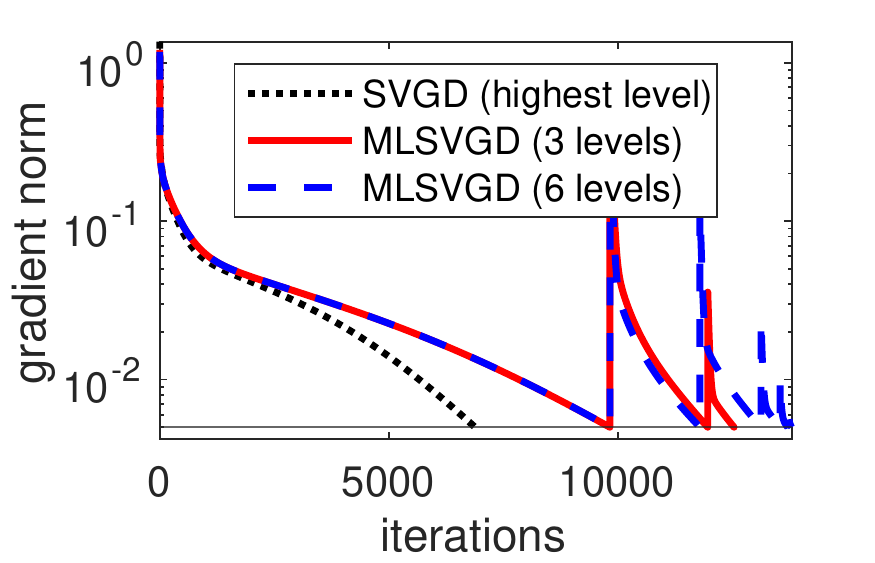} & \hspace*{-0.2cm}\includegraphics[width=0.33\columnwidth]{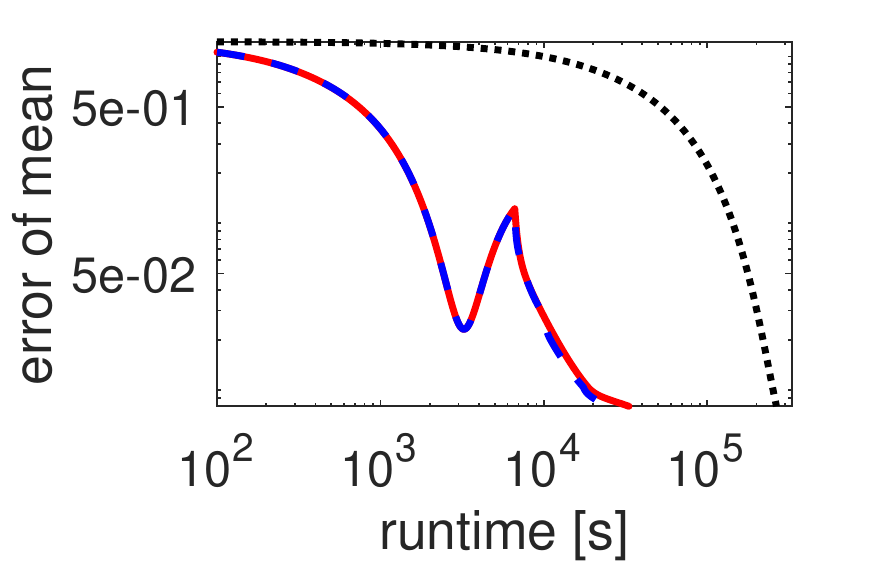}\\
(a) runtime & \hspace*{-0.2cm}(b) iterations & \hspace*{-0.2cm}(c) comparison to MCMC reference
\end{tabular}
}
\caption{Euler-Bernoulli beam: Results of MLSVGD for dimension $d = 12$.}
\label{fig:EBD12}
\end{figure}

\begin{figure}
\scalebox{0.94}{
\begin{tabular}{ccc}
\includegraphics[width=0.33\columnwidth]{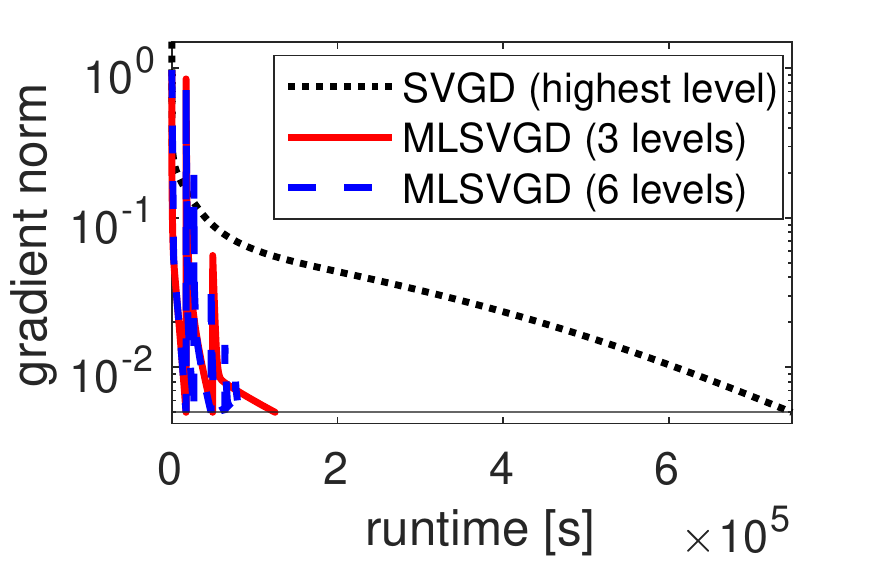} & \hspace*{-0.2cm}\includegraphics[width=0.33\columnwidth]{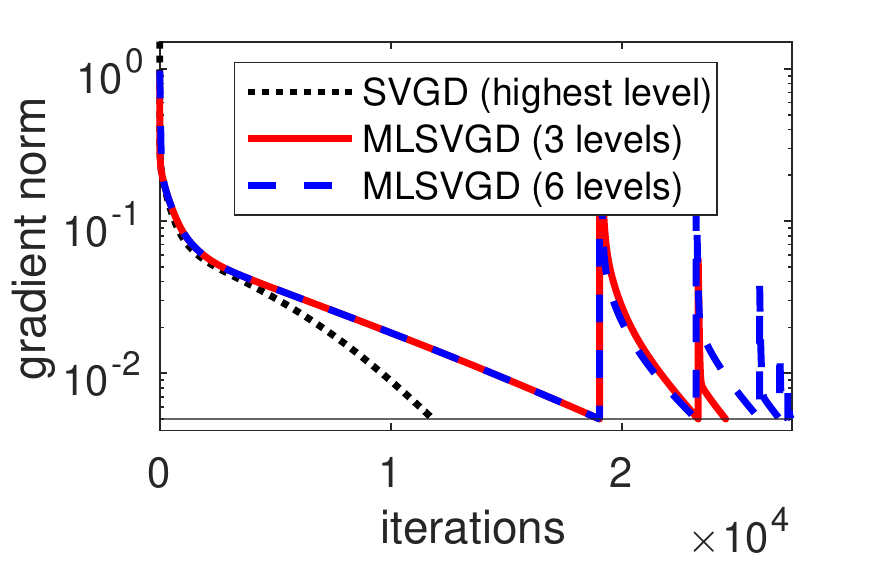} & \hspace*{-0.2cm}\includegraphics[width=0.33\columnwidth]{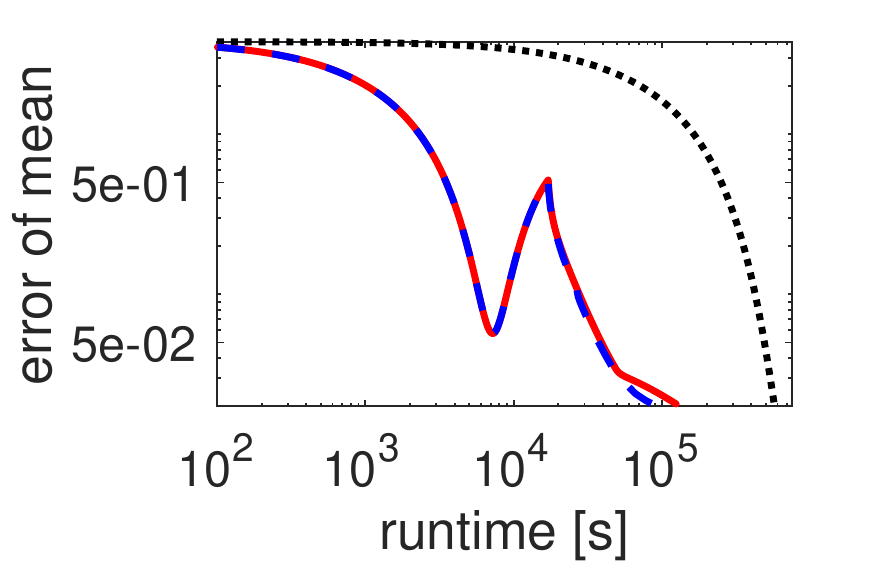}\\
(a) runtime & \hspace*{-0.2cm}(b) iterations & \hspace*{-0.2cm}(c) comparison to MCMC reference
\end{tabular}
}
\caption{Euler-Bernoulli beam: Results of MLSVGD for dimension $d = 16$.}
\label{fig:EBD16}
\end{figure}

\begin{figure}
\scalebox{0.94}{
\begin{tabular}{ccc}
\includegraphics[width=0.33\columnwidth]{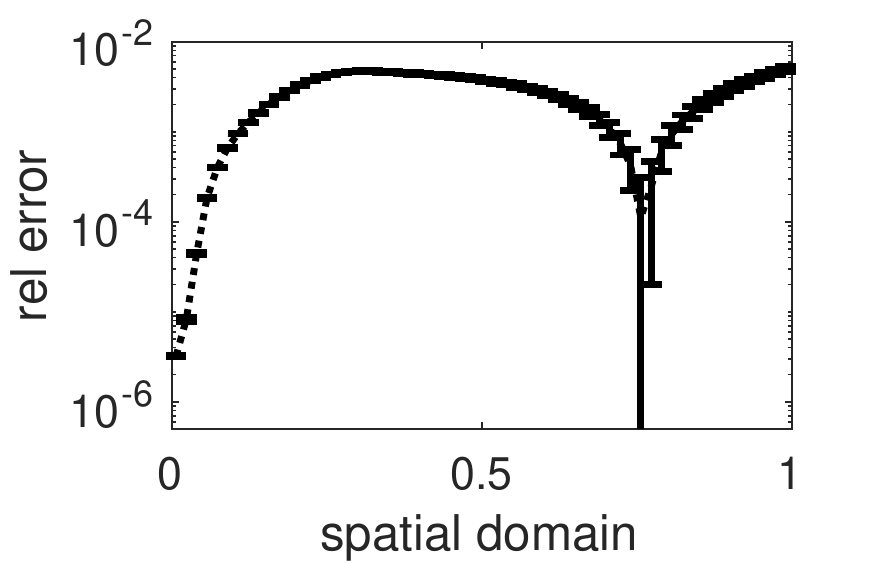} & \hspace*{-0.2cm}\includegraphics[width=0.33\columnwidth]{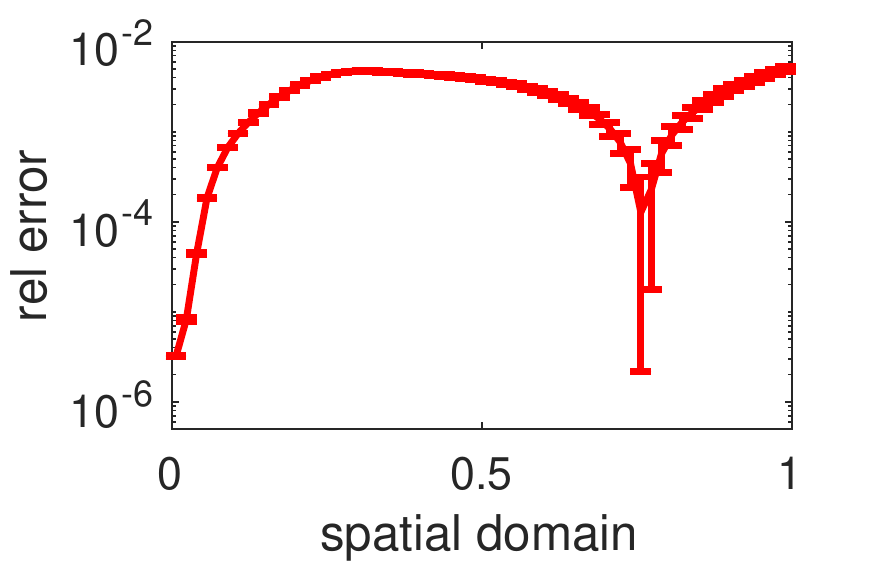} & \hspace*{-0.2cm}\includegraphics[width=0.33\columnwidth]{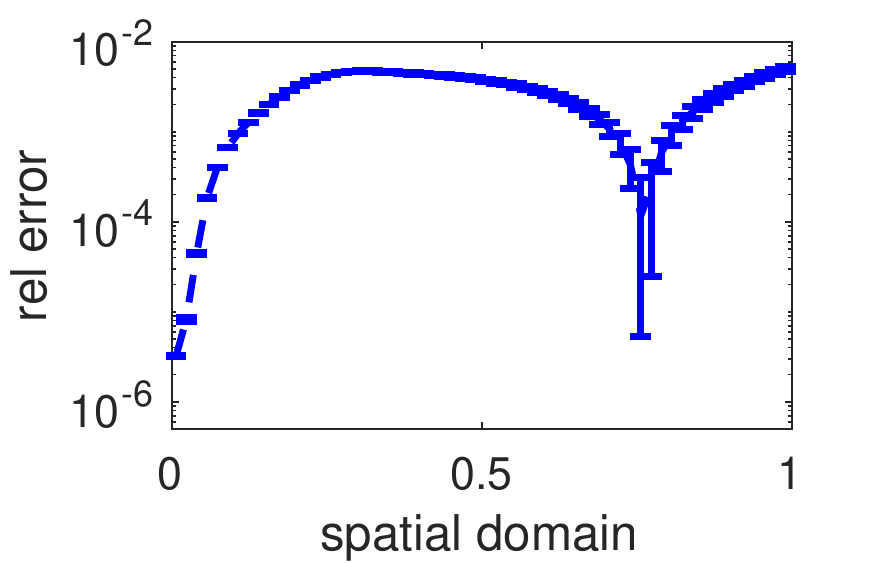}\\
(a) $d = 3$, single-level SVGD & \hspace*{-0.2cm} (b) $d = 3$, MLSVGD (3 levels) & \hspace*{-0.2cm}(c) $d = 3$, MLSVGD (6 levels)\\
\includegraphics[width=0.33\columnwidth]{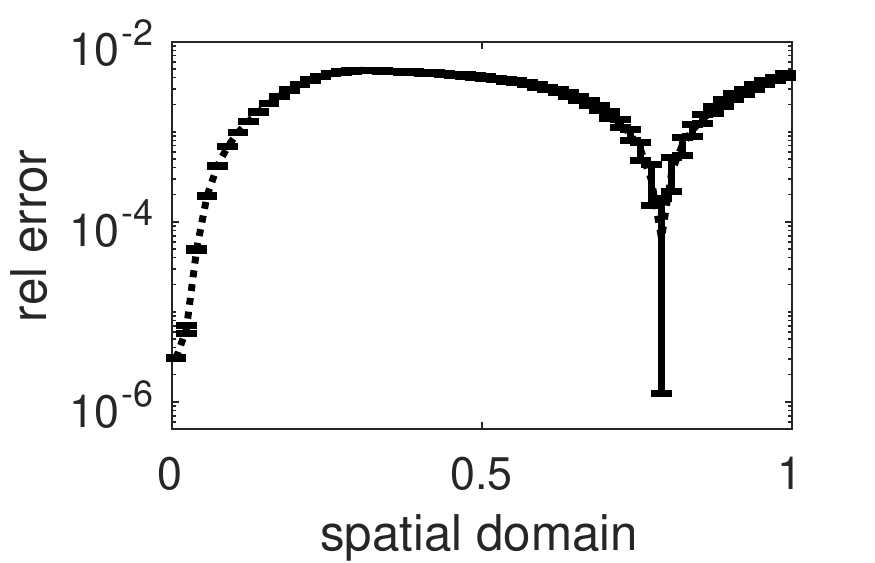} & \hspace*{-0.2cm}\includegraphics[width=0.33\columnwidth]{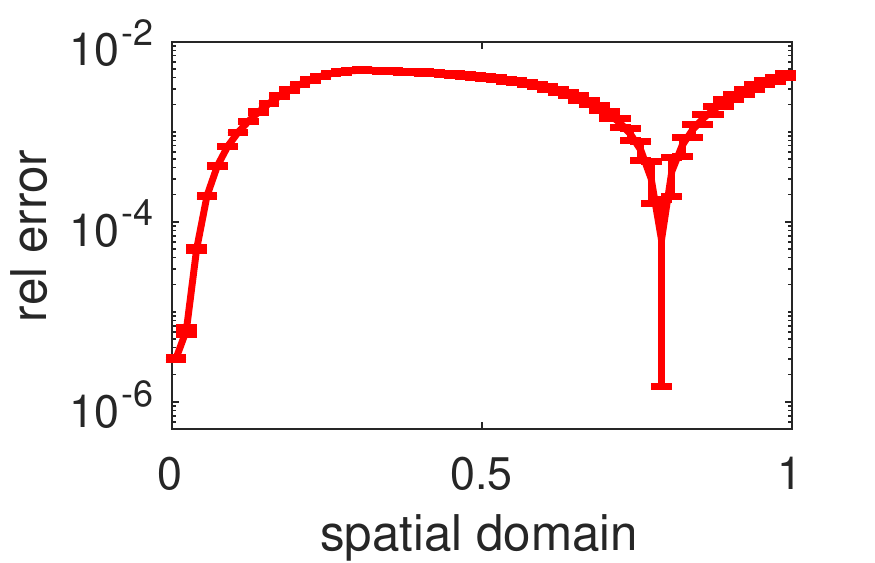} & \hspace*{-0.2cm}\includegraphics[width=0.33\columnwidth]{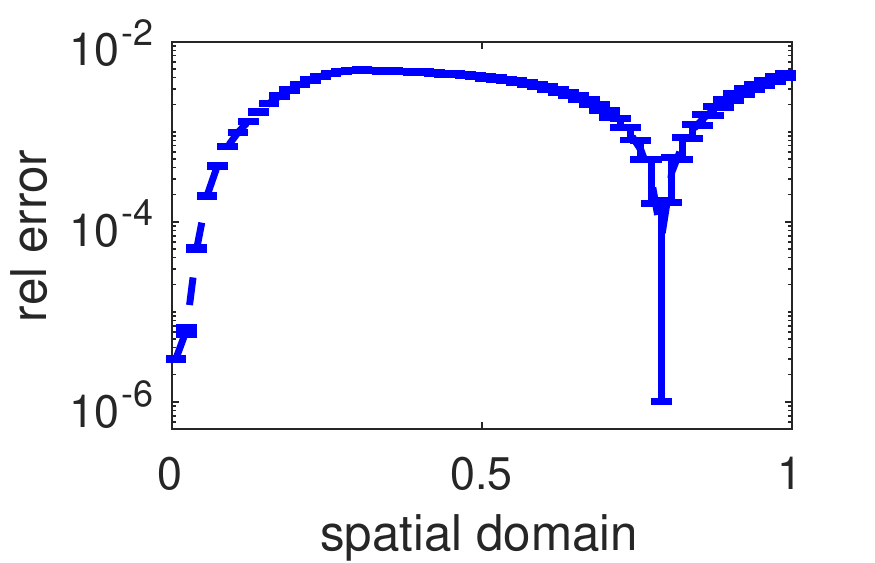}\\
(d) $d = 6$, single-level SVGD & \hspace*{-0.2cm} (e) $d = 6$, MLSVGD (3 levels) & \hspace*{-0.2cm}(f) $d = 6$, MLSVGD (6 levels)\\
\includegraphics[width=0.33\columnwidth]{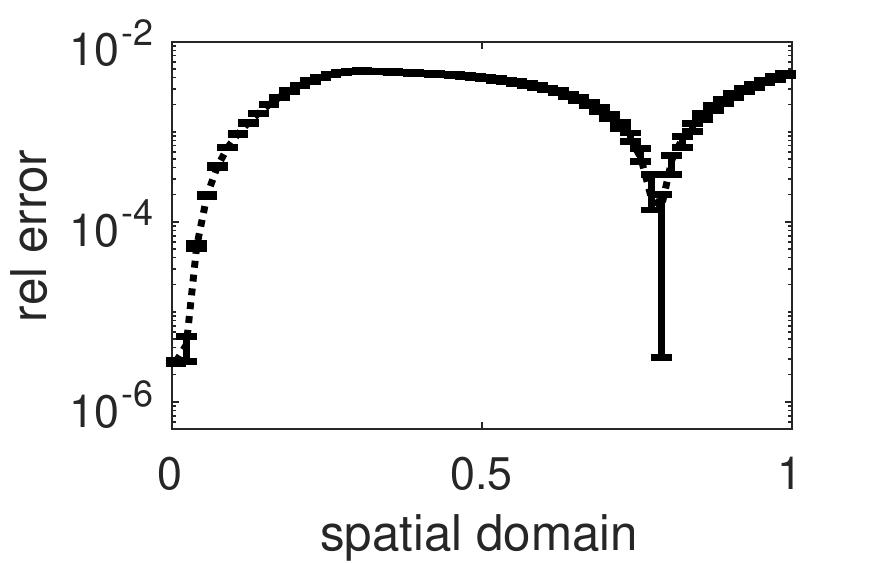} & \hspace*{-0.2cm}\includegraphics[width=0.33\columnwidth]{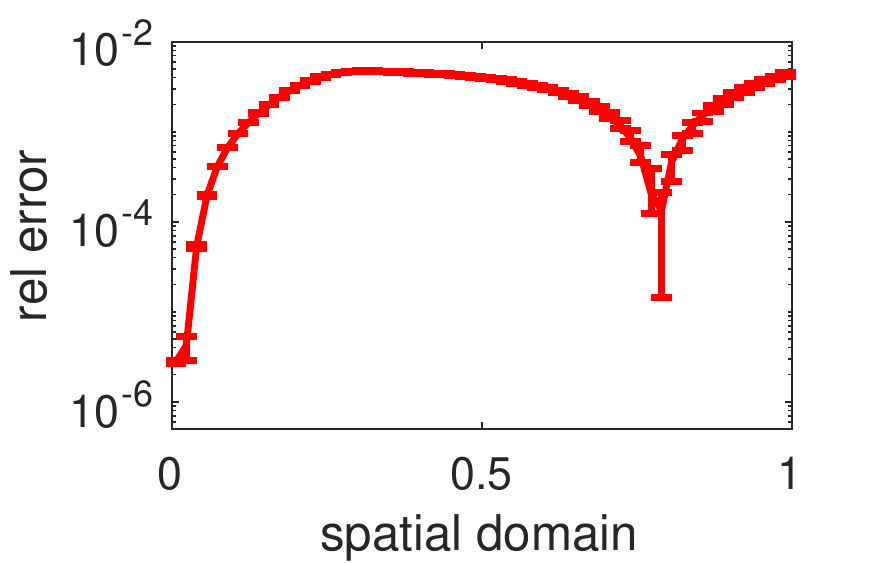} & \hspace*{-0.2cm}\includegraphics[width=0.33\columnwidth]{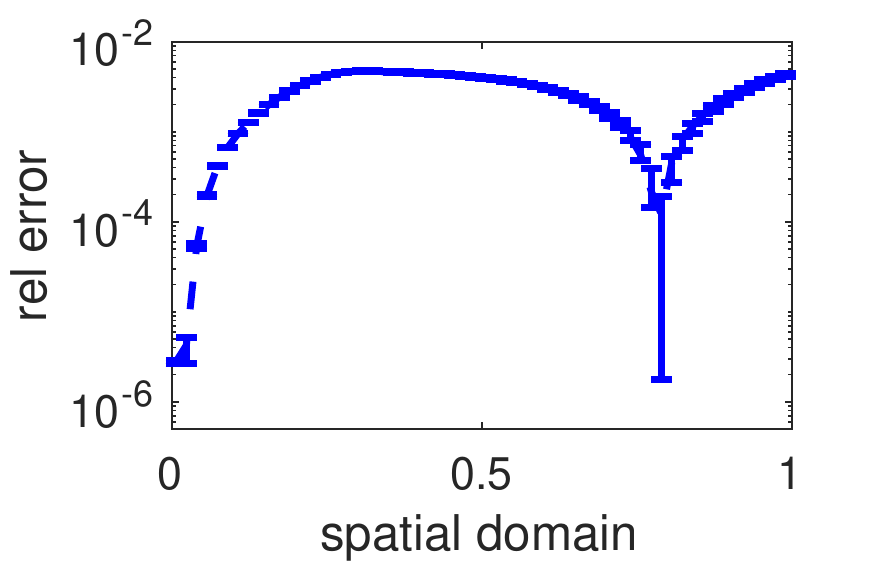}\\
(g) $d = 12$, single-level SVGD & \hspace*{-0.2cm} (h) $d = 12$, MLSVGD (3 levels) & \hspace*{-0.2cm}(i) $d = 12$, MLSVGD (6 levels)\\
\includegraphics[width=0.33\columnwidth]{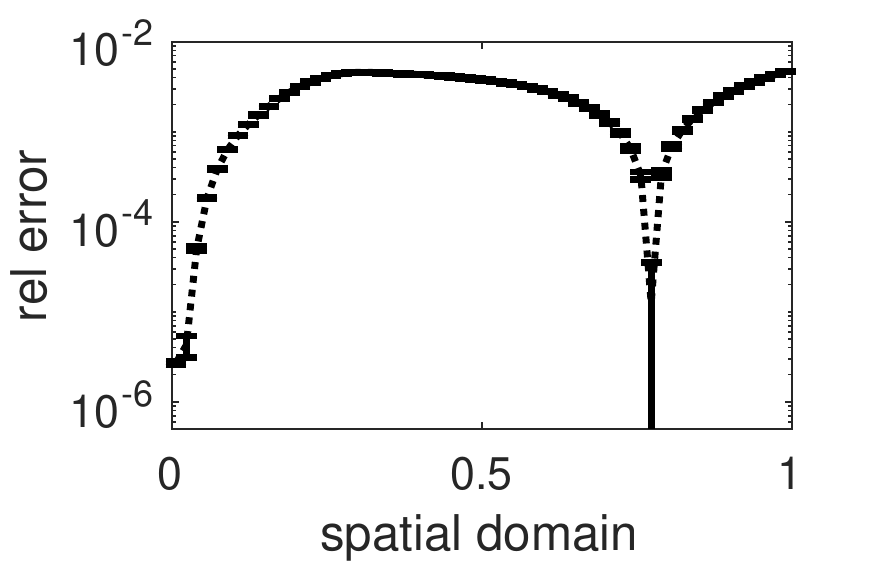} & \hspace*{-0.2cm}\includegraphics[width=0.33\columnwidth]{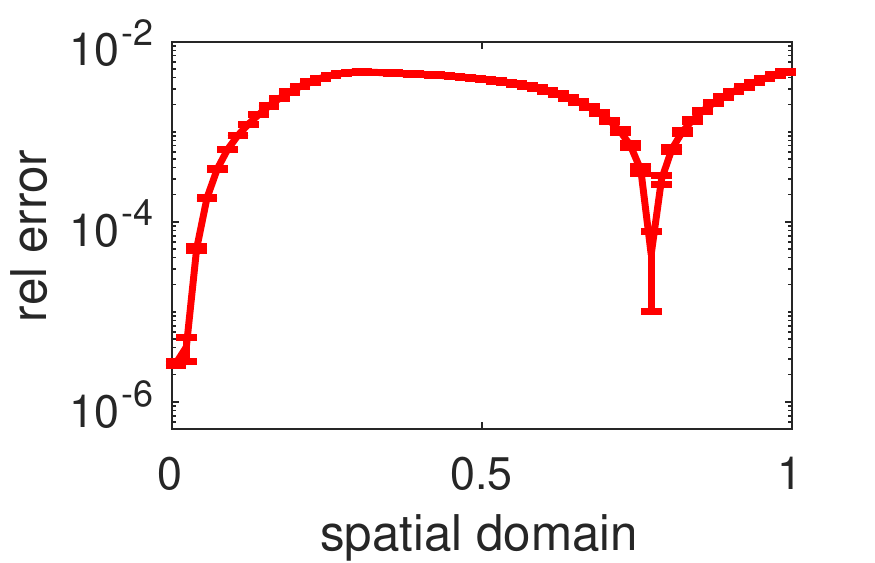} & \hspace*{-0.2cm}\includegraphics[width=0.33\columnwidth]{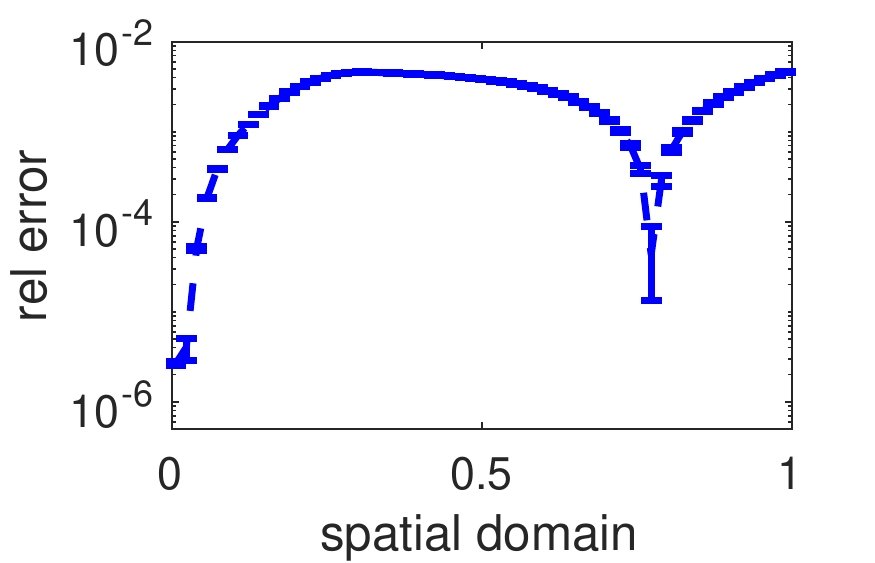}\\
(j) $d = 16$, single-level SVGD & \hspace*{-0.2cm} (k) $d = 16$, MLSVGD (3 levels) & \hspace*{-0.2cm}(l) $d = 16$, MLSVGD (6 levels)\\
\end{tabular}
}
\caption{Euler-Bernoulli beam: Minimum and maximum of pointwise error over ensemble of inferred solutions for $d \in \{3, 6, 12, 16\}$.}
\label{fig:EBVariance}
\end{figure}

\end{document}